\documentclass[a4paper,11pt, reqno]{amsart}
\usepackage[DIV=12, oneside]{typearea} 
\usepackage[utf8]{inputenc}
\usepackage[T1]{fontenc}
\usepackage[english]{babel}
\usepackage{esint}
\usepackage{amssymb, amsthm, bbm}
\usepackage{enumerate, mathrsfs, hyphenat, graphicx}
\usepackage[font=sf, labelfont=sf]{caption}
\usepackage{thmtools, thm-restate}
\usepackage{thm-patch, thm-kv, thm-autoref}
\usepackage[colorlinks=true, urlcolor=blue, linkcolor=blue, citecolor=black]{hyperref}
\theoremstyle{plain}
\declaretheorem[title=Theorem, parent=section]{sa}
\declaretheorem[title=Lemma,sibling=sa]{lem}
\declaretheorem[title=Corollary,sibling=sa]{cor}
\declaretheorem[title=Proposition,sibling=sa]{prop}
\newtheorem*{thm*}{Satz}
\theoremstyle{definition}

\declaretheorem[unnumbered,title=Remark]{bem}

\numberwithin{equation}{section}
\newcommand{\R}{  \mathbb{R}}

\newcommand{\Q}{  \mathbb{Q}}
\newcommand{\N}{  \mathbb{N}}
\renewcommand{\P}{ \mathbb{P}}

\renewcommand{\epsilon}{\varepsilon}

\newcommand{\norm}[1]{\left\| #1 \right\|}
\newcommand{\bet}[1]{\left| #1 \right|}
\newcommand{\ska}[1]{\left\langle #1 \right\rangle}

\newcommand{\Ind}[1]{\ensuremath \mathbbm{1}_{#1}}

\addto\extrasenglish{
\renewcommand{\sectionautorefname}{Section}%
}

\newtheorem{thm}{Theorem}[section]

\newtheorem{note}[thm]{Note}
\newtheorem{ex}[thm]{Example}

\newtheorem{rem}[thm]{Remark}
\def\A{{\mathcal A}}

\def\M{{\mathcal M}}
\def\P{{\mathbb P}}
\def\S{{\mathcal S}}

\def\real{{\mathbb R}}

\def\sphere{{\mathbb{S}}}
\def\Indicator{{\mathbbm{1}}}

\def\ep{\varepsilon}
\def\al{\alpha}
\def\del{\delta}
\def\om{\omega}

\def\gam{\gamma} 
\def\Gam{\Gamma} 
\def\lam{\lambda}
\def\Lam{\Lambda}

\def\grad{\nabla}

\def\union{\bigcup}
\def\intersect{\bigcap}

\newcommand{\abs}[1]{\left| #1 \right|}

\usepackage[usenames,dvipsnames]{xcolor} 

\begin{document}\renewcommand{\sectionautorefname}{Section}
\renewcommand{\sectionautorefname}{Section}
\title{H\"older Regularity For Integro-Differential Equations With Nonlinear Directional Dependence}
\author{Moritz Kassmann}
\author{Marcus Rang}
\author{Russell W. Schwab}

\begin{abstract}
We prove H\"older regularity results for a class of nonlinear elliptic integro-differential operators with integration kernels whose ellipticity bounds are strongly directionally dependent.  These results extend those in \cite{CaSi-09RegularityIntegroDiff} and are also uniform as the order of operators approaches 2.
\end{abstract}

\address{Fakultät für Mathematik\\
Universität Bielefeld \\
Postfach 100131 \\
D-33501 Bielefeld}
\email{moritz.kassmann@uni-bielefeld.de}

\address{Fakultät für Mathematik\\
Universität Bielefeld \\
Postfach 100131 \\
D-33501 Bielefeld}
\email{mrang@math.uni-bielefeld.de}

\address{Department of Mathematics\\
Michigan State University\\
619 Red Cedar Road \\
East Lansing, MI 48824}
\email{rschwab@math.msu.edu}

\thanks{\today. We would like to thank L. Silvestre for helpful discussions concerning \cite{CaSi-09RegularityIntegroDiff} as well as bringing some important references to our attention. }

\maketitle

\pagestyle{headings}		
\markboth{}{Non-Degenerate Kernels}

\section{Introduction}\label{sec:intro}

In this note we prove some H\"older regularity results for viscosity solutions of integro-differential equations in which the kernels defining the operators have strong directional dependence and do not need to satisfy everywhere pointwise comparison with the canonical kernel corresponding to the fractional Laplacian.  One of our aims is to bring into better alignment the results which have been known for linear equations from the Probability and Potential Theory communities and those from the Nonlinear PDE community (see \autoref{sec:Background}). 

Before proceeding we mention that we have tried to collect the notation contained herein-- as much as possible-- in Section \ref{Sec:Notation}.  We also include a more detailed discussion of background in Section \ref{sec:Background}.  We first state our main result, and then develop the related operators and background in the remainder of \autoref{sec:intro} and \autoref{sec:Background}.

The simplest example of the operators we study is 

\begin{equation}\label{IntroEq:LinearK}
	Lu(x) = \int_{\real^n} \del^2_h u(x) K(x,h) dh,
\end{equation}
where $\del^2_h u(x)=u(x+h)+u(x-h)-2u(x)$.  The most canonical choice is 
\[
K(x,h) = C(n,\al)\abs{h}^{-n-\al},
\]
which for an appropriate constant $C(n,\al)$ gives $L = -(-\Delta)^{\al/2}$, where $-(-\Delta)^{\al/2}$ is the operator whose Fourier multiplier is $-\abs{\xi}^\al$ (see \cite[Chapter I.1]{Land-72}).
The interesting cases we target are when $K$ is allowed to have large regions where $K(x,h)$ is not necessarily comparable to $\abs{h}^{-n-\al}$ from below-- instead, $K(x,\cdot)$, is only required to be in the class we call $\A_{sec}$, see (\ref{PrelimNotationEq:ClassAsect}).  Furthermore we treat the case of $L$ without assuming any regularity in the $x$ variable.  

The nonlinear directional dependence enters the picture through the assumption that the kernels we treat need only to be above $\abs{h}^{-n-\al}$ on a possibly small set as seen by
\begin{align}\label{specialK}
(2-\alpha)\Ind{V_\xi}(h)\frac{\lambda}{|h|^{n+\alpha}}\leq K(x,h)\leq (2-\alpha) \frac{\Lambda}{|h|^{n+\alpha}} \quad h\in\R^n\setminus\{0\}.
\end{align}
Here $V_\xi \subset \R^n$ is a conical set of the form $V_\xi = \left\{z\in \R^n | \, \bet{\ska{z/\bet{z},\xi}}\geq \delta\right\}$ with $\xi \in \mathbb{S}^{n-1}$, and $\del\in(0,1)$ is fixed throughout.  The direction $\xi$ is allowed to depend on $x$ and $u$, and hence the title of this article.  The flexibility of $\xi$ to depend on $x$ and $u$ presents significant challenges to proving our main result. 

It is by now standard that to study regularity properties of solutions to equations involving $L$ assuming no regularity in $x$, one must in some sense ``treat all equations at once'', see  \cite[Sections 2, 3]{CaSi-09RegularityIntegroDiff}, \cite[Chapter 2]{CaCa-95}, \cite[Chapters 9, 17]{GiTr-98}, \cite{Krylov-1980ControlledDiffusionProcesses}.  This means that rather than study directly solutions of 
\[
Lu(x) = f(x)\ \textnormal{in}\ \Omega,
\]
for $f\in L^\infty(\Omega)$,
we instead study $u$ which simultaneously solve the two inequalities
\[
\inf_{K\in\A}\left\{L_Ku(x)\right\}\leq C\ \text{and}\ \sup_{K\in\A}\left\{L_Ku(x)\right\} \geq -C\ \textnormal{in}\ \Omega.
\]
The class of kernels $\A$ is chosen (in our case, described by (\ref{PrelimNotationEq:ClassAsect})), so that it will at least contain all the $K$ under consideration (and sometimes is a much larger set if one wishes to attain further convenient properties of the extremal operators, e.g. rotational invariance).  The extremal operators are given as
\[
M^-_\A u(x) = \inf_{K\in\A}\left\{L_Ku(x)\right\}\ \text{respectively}\ M^+_\A u(x) = \sup_{K\in\A}\left\{L_Ku(x)\right\}.
\]
By this line of argument, treating the case of $L$ with only bounded measurable dependence in $x$ is basically the same as treating general fully nonlinear equations
\begin{equation}\label{IntroEq:FullyNonlinear}
F(u,x)=f(x),
\end{equation}
as soon as $F$ satisfies the ellipticity assumption
\begin{equation}\label{IntroEq:Ellipticity}
M^-_\A(u-v)(x) \leq F(u,x)-F(v,x) \leq M^+_\A(u-v)(x).
\end{equation}
(see \cite[Sections 2, 3]{CaSi-09RegularityIntegroDiff},  \cite{CaCa-95}).  
We can now see that if a $K([u],x,\cdot)$ is chosen as an optimizer of one of the extremal operators at each $x$, then again 
\begin{equation}\label{IntroEq:DirectionalDependence}
	Lu(x) = \int_{\real^n} \del^2_h u(x) K([u],x,h) dh
\end{equation}
would fall into this same class of bounded measurable coefficients.  This reinforces the notion of nonlinear directional dependence, depending on the unknown, $u$.

The program of studying regularity properties of nonlocal equations such as (\ref{IntroEq:FullyNonlinear}) was presented in \cite{CaSi-09RegularityIntegroDiff}, and here we extend those results to cover the larger class, $\A_{sec}$.  Our main result is the following Hölder regularity estimate.

\begin{sa}[H\"older Regularity]\label{thm:holder}
Let $M^\pm_{\A_{sec}}$ be as defined in (\ref{PrelimNotationEq:MPlusDef}) and (\ref{PrelimNotationEq:MMinusDef}), and let $\al\in(\al_0,2)$.  There are positive constants $\beta \in (0,1)$ and $C \geq 1$ depending only on $n,\lambda,\Lambda,\alpha_0$ and $\delta$ such that if $u\in L^{\infty}(\real^n)$ satisfies in the viscosity sense
\begin{equation}\label{IntroEq:ExtremalBounded}
	M^-_{\A_{sec}}u\leq C'\ \ \text{and}\ \ M^+_{\A_{sec}}u \geq -C'\ \ \text{in}\ \ B_1(0),
\end{equation}
then
\begin{align}\label{hoelderinq}
\norm{u}_{C^\beta(B_{1/2}(0))}\leq C\Big( \sup\limits_{\R^n} \bet{u} + C'\Big).\end{align}
Furthermore $C$ remains bounded $\al\to 2^-$.
\end{sa}

\begin{rem}
	Just as in \cite[Theorem 26]{CaSi-09RegularityByApproximation}, Theorem \ref{thm:holder} also applies to any $u$ such that
	\[
	u(y)(1+\abs{y}^{n+\al_0})^{-1} \in L^{1}(\real^n).
	\]
\end{rem}

\begin{rem}
	An important application of \autoref{thm:holder} is the $C^{1,\beta'}(B_{1/2})$ regularity of solutions of (\ref{IntroEq:FullyNonlinear}) with $f=0$ and a \emph{translation invariant} $F$ satisfying (\ref{IntroEq:Ellipticity}) for a special subclass $\tilde \A_{sec}\subset \A_{sec}$.  Just as in \cite{CaSi-09RegularityIntegroDiff} we impose the additional restriction that for all $K\in\tilde\A_{sec}$ it holds for a fixed $\rho_0>0$ and a uniform $\Tilde C$
	\begin{equation}\label{IntroEq:KernelDifferences}
		\int_{\real^n\setminus B_{\rho_0}}\frac{\abs{K(h)-K(h-z)}}{\abs{z}}dh\leq \Tilde C\ \text{for each}\ \abs{z}<\frac{\rho_0}{2}.
	\end{equation} 
	Then \cite[Theorem 13.1]{CaSi-09RegularityIntegroDiff} carries over to our setting with almost no modifications, and we conclude that also \autoref{thm:holder} implies $C^{1,\beta'}$ regularity for this larger class, $\Tilde\A_{sec}$.  The proof works because \cite[Theorem 13.1]{CaSi-09RegularityIntegroDiff} only uses two main ingredients involving the kernels: the assumption (\ref{IntroEq:KernelDifferences}) and the ellipticity properties of $F$, $M^\pm_{\Tilde\A_{sec}}$.  See also \cite{Kriventsov-2013RegRoughKernelsArXiv} for more general results on $C^{1,\beta'}$ regularity.
\end{rem}

\begin{rem}
	Although we motivated Theorem \ref{thm:holder} using some operators with $x$ dependence, the uniqueness issue for viscosity solutions involving operators as (\ref{IntroEq:LinearK}), (\ref{IntroEq:FullyNonlinear}) is still open.  So far uniqueness is known for translation invariant operators like (\ref{IntroEq:FullyNonlinear})-- \cite[Sections 2-5]{CaSi-09RegularityIntegroDiff}-- and ones which can be written in the so-called L\'evy-Ito form in \cite[Section 2.2]{BaIm-07}.  Typically the L\'evy-Ito form can be rewritten as (\ref{IntroEq:LinearK}).  
\end{rem}

\begin{rem}
	Just as in \cite{CaSi-09RegularityIntegroDiff}, additional difficulty arises from finding a proof of Theorem \ref{thm:holder} in which $C$ remains bounded as $\al\to2^-$.  In this article these difficulties are mostly contained in Sections \ref{sec:ABPIntegro} and \ref{sec:SpecialFunc}.
\end{rem}

\begin{rem}
	In the case that $V_\xi=\real^n$ in (\ref{specialK}),  \cite[Theorem 12.1]{CaSi-09RegularityIntegroDiff} is contained is \autoref{thm:holder}. We use the same methods employed therein-- even some of the statements of the auxiliary results are the same between \cite{CaSi-09RegularityIntegroDiff} and this work.  However, there are significant technical difficulties which arise due to the lower bound in (\ref{specialK}) holding only on a small set, $V_\xi$.  These difficulties are hidden in the proofs of some of the auxiliary lemmas, and so we believe it is important to give a careful presentation of where the difficulties arise and how they are resolved.  The lower bound causes minor changes to the pointwise evaluation property of subsolutions (\autoref{prop:PointwiseEvaluation}) and the ABP substitute (Lemmas \ref{lem:rings}, \ref{lem:geo_concave}, \autoref{thm:cubecover}), and major changes to the construction of the special bump function (all of Section \ref{sec:SpecialFunc}) and the validity of the Harnack inequality (see \autoref{sec:HarnackFail}).  The assumptions and main result of our work are also studied in \cite{BCF12}.  Although the class $\A_{sec}$ is covered in \cite{BCF12}, we believe our presentation of the same results is a useful contribution to the field.  For example, as pointed out here in \autoref{sec:Background} and elaborated in \autoref{sec:HarnackFail}, a strong version of the Harnack inequality as claimed by \cite[Theorem 3.14]{BCF12} fails to hold when the kernels are in \ref{PrelimNotationEq:ClassAsect}. 
\end{rem}

The organization of the article is as follows.  In Section \ref{sec:Background} we review some background related to Theorem \ref{thm:holder}.  In Section \ref{sec:prelim} we collect notation, definitions, and preliminary results regarding (\ref{IntroEq:ExtremalBounded}).  Section \ref{sec:ABPIntegro} is dedicated to proving a nonlocal finite cube substitute for the Aleksandrov-Bakelman-Pucci estimate-- arguably the core of most of the regularity theory for nonlinear equations.  Section \ref{sec:SpecialFunc} is used to construct a special bump function which is crucial to the ``point-to-measure'' estimates.  In Section \ref{sec:PointEsti} we prove the point-to-measure estimates and put together the remaining pieces of the proof.  Finally in Section \ref{sec:Appendix} we present some examples, further results, and further discussion.

\section{Background}\label{sec:Background}

There is a rapidly growing collection of results related to \autoref{thm:holder}.  We will try to focus on the type of results which only depend on the ellipticity constants, $\lam$ and $\Lam$, as well as possibly the order, $\al$, and we emphasize that the list of references presented is not exhaustive.  As one obvious omission, we do not discuss related results for nonlocal Dirichlet forms or divergence form equations.  A discussion of results in these directions can be found in \cite[Section 2]{KassSchwa-2013RegularityNonlocalParaArXiv}.

There are a few interesting distinctions to be made: whether or not $K(x,h)$ is assumed to be even in $h$; whether or not the corresponding equations are linear; whether or not a Harnack inequality holds; and if the methods are probabilistic or PDE.  Especially when it is not assumed that $K(x,-h)=K(x,h)$, \autoref{thm:holder} does not apply to these equations (specified by such $K$).  This non-symmetric class represents an important area for applications.  Results obtained by probabilistic methods in many cases only capture regularity and/or Harnack's inequality for solutions of
\begin{equation}\label{BackgroundEq:harmonic}
Lu(x) = 0\ \text{in}\ B_1.
\end{equation}
In contrast, PDE techniques can usually capture the same behavior for e.g.
\[
Lu(x)\ \text{bounded in}\ B_1.
\]

A more general form of (\ref{IntroEq:LinearK}) is
\begin{equation}\label{BackgroundEq:LinearKnonsym}
	L u(x) = \int_{\R^n}\left(u(x+h)-u(x)-\langle
\nabla u(x), h \rangle \,\mathbbm{1}_{\{|h|\leq 1\}}\right)K(x,h) dh,
\end{equation}
which is usually the natural (non-divergence) form of a general integro-differential operator.  The reduction to (\ref{IntroEq:LinearK}) results from the extra assumption $K(x,h)=K(x,-h)$.

Regularity results (such as \autoref{thm:holder}) as well as the Harnack inequality for linear equations with operators similar to (\ref{IntroEq:LinearK}) obtained by probabilistic methods go back at least to \cite{BaLe02}. There (\ref{BackgroundEq:harmonic}) is treated assuming $K$ is even in $h$ and $V_\xi=\real^n$ in (\ref{specialK}).  H\"older regularity as well as a Harnack inequality are obtained. The results of \cite{BaLe02}-- both regularity and Harnack inequality-- were generalized by \cite{SoVo04}, and the regularity was generalized to variable order situations in \cite{BaKa-05Holder}.  These results were generalized in \cite{MiKa12}, where regularity results were obtained for a kernel with a lower bound just as in (\ref{specialK}).  Finally, higher regularity in the form of Schauder type estimates were obtained in \cite{Bass2009-RegularityStableLikeJFA}.  None of these results are robust as $\al\to2^-$.

In the realm of PDE methods, an important result for H\"older regularity is \cite{silvestre}, where \autoref{thm:holder} is proved for kernels which (\ref{specialK}) holds with $V_\xi=\real^n$, but no symmetry assumptions are made on $K$, and also variable order operators are included.  In \cite{BaChIm-11Holder} regularity results were obtained for kernels very similar to (\ref{specialK}) by a completely different approach now called the ``Ishii-Lions method'' which follows \cite{IshiiLions-1990ViscositySolutions2ndOrder}.  The Ishii-Lions method is quite versatile, and applies more easily to different types of equations other than just uniformly elliptic ones.  However, the results in \cite{BaChIm-11Holder} are not robust in $\al$ and they also depend on a modulus of continuity of $K$ in $x$-- as opposed to only depending on the lower and upper bounds for $K$-- which is often not desirable for applications of \autoref{thm:holder} (e.g. $C^{1,\beta'}$ regularity, \cite[Section 13]{CaSi-09RegularityByApproximation}; homogenization, \cite{Schw-10Per}, \cite{Schw-12StochCPDE}).

The first robust results appeared in \cite{CaSi-09RegularityIntegroDiff}, \cite{CaSi-09RegularityByApproximation}, \cite{CaSi-09EvansKrylov}, where a whole program was developed involving $C^{1,\beta'}$ and then classical regularity  (the Evans-Krylov Theorem for the integro-differential setting) for fully nonlinear equations which captures all of the existing second order theory as a limit $\al\to2^-$.  This led to a surge in related results, and we mention a few.  An important class of kernels are those for which the symmetry $K(x,-h)=K(x,h)$ is not assumed to hold; the results of \cite{CaSi-09RegularityIntegroDiff} were extended to the non-symmetric case in \cite{Chan-2012NonlocalDriftArxiv}, \cite{ChDa-2012NonsymKernels}.  Via \cite[Sections 10, 12]{CaSi-09RegularityIntegroDiff}, H\"older regularity for a smaller-- but different class-- of kernels than those treated in \cite{CaSi-09RegularityIntegroDiff} follows as a straightforward consequence to the Aleksandrov-Bakelman-Pucci type result obtained in \cite{GuSc12}.  The subclass of (\ref{specialK}) in which the direction $\xi$ is fixed for all $K$ was treated in \cite{RangThesis}.  This restriction actually makes the construction of a special bump function harder due to the non-rotational-invariance of the extremal operators in that instance.

An interesting feature of nonlocal equations such as (\ref{BackgroundEq:harmonic}) or (\ref{IntroEq:FullyNonlinear}) is that H\"older regularity and Harnack inequality no longer appear as a joined pair of results-- in contrast to the local ($\al=2$) theory.  We record it here for later reference:

\begin{note}[Harnack Inequality]\label{BackgroundNote:Harnack}
	The operators $L_K$ or $M^\pm$  (for a generic class, $\A$, not necessarily $\A_{sec}$) are said to satisfy the Harnack inequality if there exists a universal positive constant $c$ such that for any globally non-negative $u$ solving (\ref{BackgroundEq:harmonic}), respectively (\ref{IntroEq:ExtremalBounded}) in $B_1$, then 
	\begin{equation}
		u(x_1)\leq c(u(x_2) + C')\ \text{for all}\ x_1,x_2\in B_{1/2},
	\end{equation}
	where $C'=0$ for the case of (\ref{BackgroundEq:harmonic}) and $C'$ is given in (\ref{IntroEq:ExtremalBounded}) otherwise.
\end{note}

Typically in the second order case, one first proves a weak Harnack inequality (\cite[Chapter 8, 9]{GiTr-98} or the $L^\ep$ Lemma \cite[Lemma 4.6]{CaCa-95}, \cite[Theorem 10.3]{CaSi-09RegularityIntegroDiff}) and then deduces the Harnack inequality.  Then the Harnack inequality is used to prove reduction of oscillation and subsequently H\"older continuity (see \cite[Chapter 4]{CaCa-95}, \cite[Chapter 9]{GiTr-98}).  More care is needed in the nonlocal setting because it is not always true that the Harnack inequality holds.  In most of the results mentioned above for the integro-differential setting, H\"older regularity is deduced directly from the point-to-measure estimates, weak Harnack / $L^\ep$ Lemma, or uniform entrance/exit time estimates for a related stochastic process.  

A necessary and sufficient condition for solutions of (\ref{BackgroundEq:harmonic}) to satisfy the Harnack inequality is given in \cite{BoSz05}, where they also provide an example of an $L$ such that $K\in\A_{sec}$ (see (\ref{PrelimNotationEq:ClassAsect})) but the Harnack inequality fails.  We will discuss this example in slightly more detail in Section \ref{sec:Appendix}.  This is interesting because the Harnack inequality is proved for the fully nonlinear integro-differential case when $V_\xi=\real^n$ in (\ref{specialK}),  \cite[Section 11]{CaSi-09RegularityIntegroDiff}.  The Harnack inequality is also stated to hold in \cite{BCF12} for a class which contains $\A_{sec}$, which cannot be true by \cite[Theorem 1 and Example of p.148]{BoSz05}.  We will further discuss in Section \ref{sec:Appendix} where the proof of \cite[Section 11]{CaSi-09RegularityIntegroDiff} breaks down when one considers the larger class $\A_{sec}$, (\ref{PrelimNotationEq:ClassAsect}).


\section{Preliminaries}\label{sec:prelim}

\subsection{Notation}\label{Sec:Notation}
We first collect some notations which will be used throughout this article.
\allowdisplaybreaks
\begin{align}
	& \al\in(\al_0,2)\ \text{is the order of the operators}\nonumber\\
	& \del\in(0,1)\ \text{is the opening of a conical sector}\nonumber\\
	&V_\xi = \left\{z\in\real^n\ :\ \abs{\frac{z}{\abs{z}}\cdot\xi}\geq \del\right\}\label{PrelimNotationEq:Sector}\\
	&\A_{sec} = \bigg\{K:\real^n\to\real\ :\ K(-h)=K(h),\nonumber\\  
	&\ \ \ \ \ \ \ \ \ \ \ \ \ \ \text{and}\ \exists\ \xi\in\sphere^{n-1}\ \text{s.t.}\ \Ind{V}(h)\frac{\lambda(2-\alpha)}{|h|^{n+\alpha}}\leq K(h)\leq  \frac{\Lambda(2-\alpha)}{|h|^{n+\alpha}}\bigg\} \label{PrelimNotationEq:ClassAsect}\\
	&u\in C^{1,1}(x),\ \text{if}\ \exists\ v\in\real^n\ \text{and}\ A>0\ \text{s.t.}\nonumber\\ 
	&\ \ \ \ \ \ \bet{u(x+h)-u(x)-\ska{v,h}}\leq A\bet{h}^2\ \text{for h small enough}\nonumber\\
	&\del^2_hu(x) = u(x+h)+u(x-h)-2u(x)\label{PrelimNotationEq:SecondDiffDef}\\
	&L_Ku(x) = \int_{\real^n}\del^2_h u(x) K(h) dh\nonumber\\
	&M^+_{\A_{sec}}u(x) = \sup_{K\in\A_{sec}}\left\{L_Ku(x)\right\}\label{PrelimNotationEq:MPlusDef}\\
	&M^-_{\A_{sec}}u(x) = \inf_{K\in\A_{sec}}\left\{L_Ku(x)\right\}\label{PrelimNotationEq:MMinusDef}\\
	&\frac{1}{2}V_\xi = \left\{z\in\real^n\ :\ \abs{\frac{z}{\abs{z}}\cdot\xi}\geq \frac{\del+1}{2}\right\}\label{PrelimNotationEq:HalfSector}\\
	&Q_l(x_0) =\left\{x\in\R^n\, :\, \bet{x-x_0}_\infty<\tfrac{l}{2}\right\}\nonumber\\
	&tQ_l(x_0)=\left\{x\in\R^n\, :\, \bet{x-x_0}_\infty<\tfrac{tl}{2}\right\}\nonumber\\
	&B_l(x_0)=\left\{x\in\R^n\, :\, \bet{x-x_0} < l \right\}\nonumber\\
	&\mu(dh) = \abs{h}^{-n-\al}dh\nonumber\\
	&dS_r,\ dS\ \text{are respectively surface measure on the spheres}\ \partial B_r, \partial B_1\nonumber
\end{align}

We use $\bet{\cdot}$ for the absolute value, the Euclidean norm, and the $n$-dimensional Lebesgue measure at the same time. Throughout this article $\Omega\subset\R^n$ denotes a bounded domain.
For cubes and balls such that $x_0=0$ we write $Q_l$ instead of $Q_l(0)$ and similarly for $B_l$. Note that the following implications hold:
\[ B_{1/2} \subset Q_1 \subset Q_3 \subset B_{\tfrac{3\sqrt{n}}{2}} \subset B_{2 \sqrt{n}} \,. \]
Note that $2Q_1  \ne  \{2x \in\R^n: x \in Q_1 \}$.

\subsection{Definitions}
We use the definitions and basic properties of viscosity solutions from \cite[Sections 3, 4, 5]{CaSi-09RegularityIntegroDiff} and \cite{BaIm-07} for (\ref{IntroEq:ExtremalBounded}) .  The curious reader should also see the presentation and references from the local theory in \cite{CrIsLi-92}.

\subsection{Pointwise Evaluation}\label{sec:subsecPointWiseEvaluation}

A very useful feature of the viscosity solution theory for integro-differential equations is that viscosity subsolutions themselves-- not only the test functions-- can be used to evaluate their corresponding equation classically at all of the points where the equation is expected to hold in the weak sense-- that is at all points where the subsolution can be touched from above by a smooth test function.  Although this aspect of the theory is not complicated, one must proceed carefully because in the general situation of $K\in\A_{sec}$ one no longer has the convenient result from previous cases (cf. \cite[Lemma 4.3]{CaSi-09RegularityIntegroDiff}) which says that at points where $u$ is touched from above by a smooth function, the extremely strong result holds:
\begin{equation*}
	\int_{\real^n}\abs{\del^2_hu(x)}\abs{h}^{-n-\al}dh <\infty.
\end{equation*}
Although this convenient regularity on $u$ is false in the general cases treated herein (see \autoref{sec:PointwiseExamples}), the definitions of $M^+$ and viscosity subsolutions together guarantee enough regularity on $u$ so that the equation can still be evaluated classically at points of being touched by a test function.  The main result we use is Proposition \ref{prop:PointwiseEvaluation} (cf. \cite[Lemma 4.3]{CaSi-09RegularityIntegroDiff}).

Since in this work, we only need this property for the equation
\begin{equation}\label{PointwiseEq:USubsolution}
		M^+_{\A_{sec}}u(x)\geq -f(x)\ \text{in}\  B_1
\end{equation}
we only state the results as they pertain to this particular one.  Here we assume for the sake of simplicity that $f\in C(\overline{B_1})$.  We also note that the pointwise evaluation holds for general $F$ which are elliptic with respect any class which satisfies the upper bound of (\ref{specialK}).  We include further discussion on this matter in \autoref{sec:Appendix}.

\begin{prop}[Pointwise Evaluation]\label{prop:PointwiseEvaluation}
	Assume $u$ solves (\ref{PointwiseEq:USubsolution}) in the viscosity sense.  If for $\phi\in C^{1,1}(x)\cap L^\infty(\real^n)$,  $u-\phi$ has a global maximum at $x\in B_1$, then $M^+_{\A_{sec}}$ can be evaluated classically on $u$ at $x$, and $M^+_{\A_{sec}}u(x)\geq -f(x)$.
\end{prop}

Although the proof of Proposition \ref{prop:PointwiseEvaluation} is straightforward, we break it into two separate lemmas for clarity.

\begin{lem}[Extremal Formula]\label{lem:FormulaForMPlus}
	Assume $u\in C^{1,1}(x)\intersect L^\infty(\real^n)$.  Then we have the formulas
	\begin{equation}\label{PointEq:FormulaMPlus}
		M^+_{\A_{sec}}u(x) = \sup_{\xi\in\sphere^{n-1}}\left((2-\al)\int_{\real^n} \left(\Lam(\del^2_h u(x))^+ -\lam(\del^2_h u(x))^-\Indicator_{V_{\xi}}(h)\right)\mu(dh) \right)
	\end{equation}
	\begin{equation}\label{PointEq:FormulaMMinus}
		M^-_{\A_{sec}}u(x) = \inf_{\xi\in\sphere^{n-1}}\left((2-\al)\int_{\real^n} \left(\lam(\del^2_h u(x))^+\Indicator_{V_\xi}(h) -\Lam(\del^2_h u(x))^- \right)\mu(dh) \right)
	\end{equation}
\end{lem}

\begin{rem}
	\autoref{lem:FormulaForMPlus} is not absolutely necessary to prove \autoref{prop:PointwiseEvaluation}, but we think it is useful in its own right and hence include it here.
\end{rem}

\begin{rem}\label{PrelimRem:MPlusFormulaUSC}
	We note that for a fixed $u$ in the set (\ref{PointwiseEq:Dominated}), by Fatou's Lemma the map
	\[
	\xi\mapsto \int_{\real^n} \left(\Lam(\del^2_h u(x))^+ -\lam(\del^2_h u(x))^-\Indicator_{V_{\xi}}(h)\right)\mu(dh)
	\]
	is upper semi-continuous.  Hence the $\sup$ in (\ref{PointEq:FormulaMPlus}) is achieved for any such $u$.  This is not necessary in the arguments below, but useful for reference and simplification, and so we included it.
\end{rem}

\begin{lem}[Upper Semi-continuity]\label{PointwiseLem:MPlusUSC}
	For $x$ fixed, the functional 
	\begin{equation*}
		v\mapsto M^+_{\A_{sec}}(v,x)
	\end{equation*}
	is upper semicontinuous with respect to pointwise convergence in $h$ of $\del^2_h v(x)$ in the space of functions
	\begin{equation}\label{PointwiseEq:Dominated}
		\{v\ :\ v(h)\leq \phi(h)\ \text{for all}\ h\ \text{and}\ v(x)=\phi(x)\},
	\end{equation}
	for some fixed $\phi\in C^{1,1}(x)\cap L^\infty(\real^n)$.
\end{lem}

\begin{proof}[Proof of Lemma \ref{lem:FormulaForMPlus}]
	We only prove the case for $M^+_{\A_{sec}}$.  Let $\xi(u)\in\sphere^{n-1}$ be the optimizer for the right hand side of (\ref{PointEq:FormulaMPlus}), see \autoref{PrelimRem:MPlusFormulaUSC}.  We note that for $u$ fixed, the kernel,
	\begin{equation*}
		K_u(h)= (2-\al)\left(\lam\Indicator_{\{\del^2_z u(x)<0\}}(h)\Indicator_{V_{\xi(u)}}(h)+ \Lam\Indicator_{\{\del^2_z u(x)>0\}}(h)\right)\abs{h}^{-n-\al},
	\end{equation*}
	is in the set $\A_{sec}$.  Hence by the definition of $M^+_{\A_{sec}}$, (\ref{PrelimNotationEq:MPlusDef}), we see that the left side of (\ref{PointEq:FormulaMPlus}) is greater or equal to the right hand side.
	
	Now for the reverse inequality. Let $K$ be any other kernel in $\A_{sec}$, and let $\xi$ be the direction for the lower bound of $K$ in (\ref{specialK}).
	\begin{align*}
		&\int_{\real^n}\del^2_h u(x)K(h)dh=\int_{\real^n} \left((\del^2_h u(x))^+-(\del^2_h u(x))^-\right)K(h)dh\\
		&\ \ \ \leq (2-\al)\int_{\real^n} \left(\Lam(\del^2_h u(x))^+ -\lam(\del^2_h u(x))^-\Indicator_{V_\xi}(h)\right)\mu(dh)\\
		&\ \ \ \leq \sup_{\xi\in\sphere^{n-1}}\left((2-\al)\int_{\real^n} \left(\Lam(\del^2_h u(x))^+ -\lam(\del^2_h u(x))^-\Indicator_{V_\xi}(h)\right)\mu(dh)\right)
	\end{align*}
\end{proof}

\begin{proof}[Proof of Lemma \ref{PointwiseLem:MPlusUSC}]
	Assume that $\del^2_h v_m(x)\to \del^2_h v(x)$ pointwise in $h$. Let $\xi_m$ be optimizers of (\ref{PointEq:FormulaMPlus}) for $v_m$ (see \autoref{PrelimRem:MPlusFormulaUSC}), and let $\xi_0$ be any accumulation point of $\{\xi_m\}$.  We note that $v_m$ satisfying (\ref{PointwiseEq:Dominated}) implies $\displaystyle M^+_{\A_{sec}}v_m(x)\in [-\infty, C(\phi)]$, and so either an optimizing $\xi$ exists or $M^+_{\A_{sec}}v_m(x)=-\infty$, in which case we can assign $\xi_m$ as any element of $\sphere^{n-1}$.  Then we have the pointwise convergence for both
	\begin{equation*}
		(\del^2_h v_m(x))^-\Indicator_{V_{\xi_m}}(h)\abs{h}^{-n-\al}\to (\del^2_h v(x))^-\Indicator_{V_{\xi_0}}(h)\abs{h}^{-n-\al}
	\end{equation*}  
	and
	\begin{equation*}
		(\del^2_h v_m(x))^+\abs{h}^{-n-\al}\to (\del^2_h v(x,h))^+\abs{h}^{-n-\al}.
	\end{equation*}
	This implies that
	\begin{align}
		&\limsup_{m\to\infty}M^+_{\A_{sec}}(v_m,x)\nonumber\\ 
		&\  = (2-\al)\limsup_{m\to\infty} \left(-\lam\int_{\real^n} (\del^2_h v_m(x))^-\Indicator_{V_{\xi_m}}(h)\abs{h}^{-n-\al} dh + \Lam \int_{\real^n} (\del^2_h v_m(x))^+\abs{h}^{-n-\al}dh\right)\label{PointwiseEq:MPlusUSC1} \\
		&\ \leq (2-\al)\left(-\lam\int_{\real^n} (\del^2_h v(x))^-\Indicator_{V_{\xi_0}}(h)\abs{h}^{-n-\al}dh + \Lam \int_{\real^n} (\del^2_h v(x))^+\abs{h}^{-n-\al}dh\right)\nonumber\\
		&\ \leq M^+_{\A_{sec}}(v,x),\nonumber
	\end{align}
	where we applied respectively Fatou's Lemma and dominated convergence (hence using (\ref{PointwiseEq:Dominated})) to the first and second terms of (\ref{PointwiseEq:MPlusUSC1}).
\end{proof}

\begin{proof}[Proof of Proposition \ref{prop:PointwiseEvaluation}]
	We define the auxiliary functions
	\begin{equation*}
		\phi_r(y) =
		\begin{cases}
			\phi(y)\ &\text{if}\ y\in B_r(0)\\
			u(y)\ &\text{if}\ y\in \real^n\setminus B_r(0).
		\end{cases}
	\end{equation*}
	Proposition \ref{prop:PointwiseEvaluation} now follows directly from the facts that $u-\phi_r$ has a global maximum at $x$, $u$ is a viscosity subsolution of (\ref{PointwiseEq:USubsolution}), 
	$\del^2_h \phi_r(x,\cdot)\to \del^2_h u(x,\cdot)$ pointwise, and Lemma \ref{PointwiseLem:MPlusUSC}.  In particular
	\begin{equation*}
		-f(x)\leq M^+_{\A_{sec}}(\phi_r,x)
	\end{equation*} 
	and hence
	\begin{equation*}
		-f(x) \leq \limsup_{r\to0}M^+_{\A_{sec}}(\phi_r,x)\leq M^+_{\A_{sec}}(u,x).
	\end{equation*}
\end{proof}

\subsection{Continuity of Nonlinear Operators}

It is useful to know that general operators like 
\begin{equation*}
	F(u,x) = \inf_{a\in \S} \sup_{K\in\A_a}\left\{\int_{\real^n}\del^2_h u(x)K(h)dh\right\},
\end{equation*}
map $C^{1,1}(\Omega)\cap L^{\infty}(\real^n)\to C(\Omega)$, where $\Omega$ is an open domain.  A proof on this appears in \cite{CaSi-09RegularityIntegroDiff}, and their proof carries over immediately to our situation of $\A_{sec}$.  Although not exactly stated as such, the result of \cite{CaSi-09RegularityIntegroDiff} applies to very general situations.  We make no assumptions on $\A_a\subset\A$ other than those stated in the Lemma:

\begin{lem}[{\cite[Lemma 4.1]{CaSi-09RegularityIntegroDiff}}]\label{lem:ContinuityOfF}
	Assume that for each $r>0$,
	\[
	G(h) = \sup_a\sup_{K\in\A_a}\left( K(h) \right) \in L^1(\real^n\setminus B_r)
	\]
 (the $L^1$ norm can depend on $r$), and that 
	\[
	\lim_{s\to0} \int_{B_s}\abs{h}^2 G(h)dh = 0.
	\] 
	If $\phi\in C^{1,1}(\Omega)\cap L^\infty(\real^n)$, then $F(\phi,\cdot)\in C(\Omega)$.
\end{lem}


\section{A Nonlocal Replacement For The Aleksandrov-Bakelman-Pucci Estimate}\label{sec:ABPIntegro}

Pointwise estimates, typically the ABP estimate, are the cornerstone of nonlinear elliptic regularity theory-- often they are one of the few places where the equation (\ref{ABP-Eq:USubsolution}) is used.  These estimates typically link the supremum of a subsolution with a $L^p$ norm of the right hand side.  However in the integro-differential setting, it is well known that such estimates are for the most part a completely open question.  When searching for regularity results such as \autoref{thm:holder}, the full strength of the ABP estimate is not necessary, and it suffices to work with a finite cube approximation of $\displaystyle \norm{f}_{L^n}$.  This was first presented in \cite{CaSi-09RegularityIntegroDiff}, and has been successfully modified for use in the non-symmetric as well as the parabolic settings (\cite{Chan-2012NonlocalDriftArxiv}, \cite{ChDa-2012NonsymKernels}, \cite{ChDa-2012RegNonlocalParabolicCalcVar} ).  In this section we provide the necessary modifications to the ABP replacement to treat the larger class of kernels, $\A_{sec}$.  

For this section we assume that $u$ is a subsolution of the equation:

\begin{equation}\label{ABP-Eq:USubsolution}
	\begin{cases}
		M^+_{\A_{sec}}u(x)\geq -f(x) &\text{in}\ \  B_1\\
		u\leq 0 &\text{on}\ \ \real^n\setminus B_1,
	\end{cases}
\end{equation}
where $f\geq0$ and $f\in C(\overline{B_1})$.
 Let $\Gamma:\R^n \to \R$ be the concave envelope of $u^+$ in $B_3$ defined by 
\begin{align}\label{eq:conc_env}
\Gamma(x) =\begin{cases} \inf\{p(x) : p \,\, \text{is affine and} \,\, p\geq u^+ \,\,\text{in $B_3$}\}, &x\in B_3 \\ 0 , &x\in \R^n\setminus B_3 \,.\end{cases} 
\end{align}
Define the contact set in $B_1$ by $\Sigma =\{u=\Gamma\}\cap B_1$.   

The next lemma is the key tool for obtaining the nonlocal replacement for the ABP estimate. It states that for all points in $\Sigma$, there is at least one dyadic ring in which $u$ separates sub-quadratically from $\Gam$ in a uniformly sized portion of the ring.  This is just the right amount of regularity to eventually show that $\grad\Gam$ maps a ball centered at $x$ in the contact set to a uniformly comparable ball in the set of super-differentials.  Eventually we conclude with the finite cube replacement for the ABP which appears as Theorem \ref{thm:cubecover}.

\begin{lem}\label{lem:rings}
Let $\rho_0\in(0,1)$ and $r_j=\rho_0 2^{-\frac{1}{2-\alpha}-j}$ for $j\in\N_0$. For $x\in\R^n$ define the rings $R_j(x)=B_{r_j}(x)\setminus B_{r_{j+1}}(x)$ and the subsets

\begin{align}
	R_j(x,V_\xi) &= \bigl\{z\in R_j(x):\, z-x\in V_\xi\bigr\}\\
	D_j(x) &= \{h\in\R^n:\,u(x+h)<u(x)+\ska{h,\nabla\Gamma(x)}- Ar_j^2\}.\label{ABPEq:Dj}
\end{align}

There exists a constant $C_0=C_0(n,\delta,\rho_0,\lambda) \geq 1 $ such that  for every $x\in\Sigma$ and $A>0$ there is an index $j\in\N_0$ and a conical set $V_\xi$  ($\xi$ depends on $x$) with the property
\begin{align}
\label{ABPEq:UniformNonSeparationRingMain}
\bet{R_j(x,V_\xi)\cap\{z\in\R^n:\, u(z)<u(x)+\ska{z-x, \nabla\Gamma(x)}-Ar_j^2\}}\leq C_0 \bet{R_j(x,V_\xi)} \frac{f(x)}{A}.
\end{align}
Here $\nabla\Gamma(x)$ is any element of the super-differential of $\Gam$ in $B_3$ at $x$. 
\end{lem}

\begin{bem}
Note that $\nabla\Gamma(x)=\nabla u(x)$ for $x\in \Sigma$ if $u$ is differentiable at $x$. 
\end{bem}

\begin{proof}[Proof of \autoref{lem:rings}]
Let $x\in\Sigma$.
Since $u$ can be touched by a supporting hyperplane $p$ from above at $x$, \autoref{prop:PointwiseEvaluation} implies that $M^+_{\A_{sec}} u(x)$ is defined classically and \autoref{lem:FormulaForMPlus} (\autoref{PrelimRem:MPlusFormulaUSC}) guarantees that there exists a sector, $V_\xi$ (depending on $x$) such that
\[M^+_{\A_{sec}} u(x)=(2-\alpha)\int_{\R^n}(\Lambda(\del^2_h u(x))^+-\lambda (\del^2_h u(x))^- \Indicator_{V_\xi}(h))\, \mu(dh)\geq -f(x).\]
Note that if both $x+h\in B_3$ and $x-h\in B_3$, then 
\[\delta^2_h u(x)=u(x+h)+u(x-h)-2u(x)\leq p(x+h)+p(x-h)-2p(x)=2p(x)-2p(x)=0 \,.\]
Moreover, if either $x+h\not\in B_3$ or $x-h\not\in B_3$, then $x+h \notin B_1$ and $x-h \notin B_1$. Thus $u(x+h)\leq 0$ and $u(x-h)\leq 0$. Therefore $\delta^2_h u(x)\leq0$ for $h \in \R^n$. Thus 
\begin{align*}
-f(x)\leq M^+_{\A_{sec}} u(x) \leq -(2-\alpha)\int_{B_{r_0}}\lambda \Indicator_{V_{\xi}}(h) (\del^2_h u(x))^-\, \mu(dh) \,.
\end{align*}
Recall $r_0=\rho_0 2^{-1/(2-\alpha)}$. Since $\bigcup\limits_{j=0}^\infty R_j(0,V_\xi)\subset B_{r_0}(0)$ and $R_j(0,V_\xi)\cap R_l(0,V_\xi)=\emptyset$ for $j\neq l$, we obtain from the inequality above
\begin{align}\label{Rings}f(x)\geq(2-\alpha)\lambda\sum_{j=0}^\infty\int_{R_j(0,V_\xi)} \Indicator_{V_\xi}(h) (\delta^2_h u(x))^-\, \mu(dh) \,.
\end{align}
We want to estimate the integrals appearing in \eqref{Rings}. First note that for $h \in B_1$ 
\[0\leq(\delta^2_h u(x))^-=-\delta^2_h u(x)=-[u(x+h)-u(x)-\ska{h,\nabla\Gamma(x)}]-[u(x-h)-u(x)+\ska{h,\nabla\Gamma(x)}] \,.\]
Note that the two terms in the brackets above are nonpositive because of the concavity of $\Gamma$. We use the argument from above to estimate each integral in \eqref{Rings}:
\begin{align*}
\int_{R_j(0,V_\xi)} \Indicator_{V_{\xi}}(h) (\delta^2_h u(x))^-\, \mu(dh) &\geq-\int_{R_j(0,V_\xi)} \Indicator_{V_\xi}(h) (u(x+h)-u(x)-\ska{h,\nabla\Gamma(x)})\, \mu(dh). 
\end{align*}

Let us assume that the assertion of (\ref{ABPEq:UniformNonSeparationRingMain}) fails for all rings, i.e. for every $j\in\N_0$
\begin{align}\label{eq:ring_widerspruch}
 \bet{R_j(x,V_\xi) \cap D_j(x)} > C_0 \bet{R_j(x,V_\xi)} \frac{f(x)}{A}   \,,
\end{align}
where $D_j(x)$ appears in (\ref{ABPEq:Dj}). 
Then
\begin{align*}
-\int_{R_j(0,V_\xi)}& \Indicator_{V_\xi}(h) (u(x+h)-u(x)-\ska{h,\nabla\Gamma(x)})\, \mu(dh)\\
&\geq -\int_{R_j(0,V_\xi) \cap D_j(x)} \Indicator_{V_\xi}(h)(u(x+h)-u(x)-\ska{h,\nabla\Gamma(x)})\, \mu(dh) \\
&\geq  A r_j^2 \frac{1}{r_j^{n+\alpha}} \bet{R_j(0,V_\xi) \cap D_j(x)} \geq   C_0 \bet{R_j(0,V_\xi)} \frac{A f(x)}{A r_j^{n+\alpha-2}}  
=c_1 C_0 f(x) r_j^{2-\alpha} ,  
\end{align*}

where $c_1 >0$ depends on $n$ and $\delta$. Therefore we obtain
\begin{align*}
f(x)& \geq c_1 (2-\alpha)\lambda \,C_0\, f(x)\sum_{j=0}^\infty r_j^{2-\alpha} \\
&=\frac{c_1}{2}\, \rho_0^{2-\alpha}\,(2-\alpha)\lambda\, C_0\, f(x) \,\sum_{j=0}^\infty (2^{-(2-\alpha)})^j \\
&\geq \frac{c_1}{2}C_0\, f(x) \lambda\rho_0^2 \frac{2-\alpha}{1-2^{-(2-\alpha)}} \geq c_2 \ C_0\, f(x),
\end{align*}
with a positive constant $c_2$ depending on $n,\delta,\rho_0$ and $\lambda$. Note that $c_2$ is independent of $\alpha$. By choosing $C_0$ large enough, we obtain a contradiction, and hence (\ref{ABPEq:UniformNonSeparationRingMain}) holds for at least one ring.
\end{proof}

The goal of the remainder of this section is to construct a specific covering of the contact set $\{u=\Gamma\}$ by a finite number of cubes. We need the following lemma which relates a bound of concave functions in a portion of the ball to an estimate in the whole ball.

\begin{lem}\label{lem:geo_concave}
Define $R= B_1\setminus B_{1/2}$ and $R(0,V_\xi)=R\intersect V_\xi$. There exists $l=l(n,\delta)\in (0,\tfrac12)$ and $\ep_0(n,\del)>0$ such that for every concave function $G:B_1\to\R$ and $b>0$ satisfying 
\[\bet{\{z\in R(0,V_\xi):G (z)<G (0)+ \langle z ,  \nabla G(0) \rangle -b\}}\leq \ep_0 \bet{R(0,V_\xi)},\] the inequality   
\[G(y)\geq G(0)+\langle y, \nabla G(0) \rangle -b\]
holds for every $y\in B_l$.
\end{lem}

\begin{bem}
Note that the assertion of this result is weaker than the corresponding one of \cite[Lemma 8.4]{CaSi-09RegularityIntegroDiff}. The uniform estimate only capturing $y\in B_l$ is due to the geometric restriction imposed by the set $V_\xi$.  Basically $B_l$ is the largest sized ball contained in the convex hull of $\frac{1}{2}V_\xi\intersect B_{1/2}$.  See (\ref{PrelimNotationEq:HalfSector}) for $\frac{1}{2}V_\xi$.
\end{bem}

\begin{proof}

We will prove this estimate in two steps.  First we will show that in half of the sector, $ \frac{1}{2}V_\xi$ (see (\ref{PrelimNotationEq:HalfSector})), the uniform estimate holds in $B_{1/2}$.  Second we note that $l$ can be the radius of the largest ball centered at $0$ such that $ B_l\subset\textnormal{hull}(B_{1/2}\intersect \frac{1}{2}V_\xi)$.

We let $y\in B_{1/2}\intersect \frac{1}{2}V_\xi$ be generic. 
Choose $l_0\in (0,\tfrac12 )$ sufficiently small such that one can find two points $y_1$ and $y_2$ in $R(0,\frac{1}{2}V_\xi)$ such that \[y=(y_1+y_2)/2\] and both $B_{l_0}(y_1)$ and $B_{l_0}(y_2)$, are contained in $R(0,V_\xi)$.
We claim that $\ep_0$ can be chosen small enough, such that for every $y\in B_{1/2}\intersect\frac{1}{2}V_\xi$, $y_1$ and $y_2$ as above, and every $b>0$ satisfying 
\begin{align}\label{eq:GPunkt}
\bet{\{z\in R(0,V_\xi): G(z)<G (0)+\ska{z,\nabla G(0)}-b\}}\leq \ep_0 \bet{R(0,V_\xi)},
\end{align}
there will be two points $z_1\in B_{l_0}(y_1)$ and $z_2\in B_{l_0}(y_2)$ such that 
\begin{enumerate}[(i)]
 \item $y=(z_1+z_2)/2$, 
\item $G(z_1)\geq G(0)+\ska{z_1,\nabla G(0)}-b$, \,\, and
\item $G(z_2)\geq G(0)+\ska{z_2,\nabla G(0)}-b$.
\end{enumerate}
We prove the claim as follows: Choose $\ep_0$ sufficiently small such that 
\[\ep_0 \bet{R(0,V_\xi)}<\frac{\bet{B_{l_0}}}{2}.\] 
Let $y\in B_{1/2}\intersect R(0,V_\xi)$ and $y_1, y_2 \in R(0,V_\xi)$ be arbitrary with $y=(y_1+y_2)/2$.  We define
\begin{align*}D_1&=\{z_1\in B_{l_0}(y_1): G(z_1)\geq G(0)+\ska{z_1,\nabla G(0)}-b\}\subset R(0,V_\xi)\, ,\\ D_2&=\{z_2\in B_{l_0}(y_2): G(z_2)\geq G(0)+\ska{z_2,\nabla G(0)}-b\}\subset R(0,V_\xi).\end{align*}
Using \eqref{eq:GPunkt} and the choice of $\ep_0$ from above, we obtain $\bet{D_1}>\frac{\bet{B_{l_0}}}{2}$ and $\bet{D_2}>\frac{\bet{B_{l_0}}}{2}$. \\
It is clear that for every point $z_1\in B_{l_0}(y_1)$ there exists a point $z_2\in B_{l_0}(y_2)$ such that $y=\frac{z_1+z_2}{2}$. 
We want to find points $z_1\in D_1$ and $z_2\in D_2$ such that $y=\frac{z_1+z_2}{2}$. Let us assume that this is not possible. Hence, for every $z_1\in D_1$ we find a point $z_2\in B_{l_0}(y_2)\setminus D_2$ such that $y=\frac{z_1+z_2}{2}$. This implies that
\[\bet{B_{l_0}(y_2)\setminus D_2}\geq \bet{D_1}>\frac{\bet{B_{l_0}}}{2}.\]
This is a contradiction to the fact that $\bet{D_2}>\frac{\bet{B_{l_0}}}{2}$. This proves our claim.\\
For $z_1\in B_{l_0}(y_1)$ and $z_2\in B_{l_0}(y_2)$ satisfying (i)-(iii) we finally have 
\begin{align*}
G (y)&=G\left(\frac{z_1+z_2}{2}\right)\geq\frac{1}{2}G(z_1)+\frac{1}{2}G(z_2) \\
& \geq G (0)+\frac{1}{2}\ska{z_1+z_2,\nabla G(0)}-b = G(0)+\ska{y,\nabla G(0)}-b. 
\end{align*}

Now to conclude the second step, we simply remark that by concavity the bound must hold for all $y$ in the convex hull of $B_{1/2}\intersect \frac{1}{2}V_\xi$.  Thus taking $l$ to be the radius of the largest ball contained in the convex hull, we have the estimate for the decay of $G$ for all $y\in B_l$.
\end{proof}

By a simple scaling argument we get \autoref{lem:geo_concave} for every ball:
\begin{cor}
\label{Lem8.4anyball}
For $x\in\R^n$ and $r>0$ define $R_r(x)= B_r(x)\setminus B_{r/2}(x)$ and the subset \[R_r(x,V_\xi)=\bigl\{y\in\,R_r(x):\, (y-x)\in  V_\xi\bigr\}.\] 

For every concave function $G: B_r(x)\to\R$ and $b>0$ satisfying 
\begin{align}\label{Lem8.4ball}\bet{\{z\in R_r(x,V_\xi):G(z)<G (x)+\ska{z-x,\nabla G(x)}-b\}}\leq \ep_0\bet{R_r(x,V_\xi)},\end{align} the inequality   
\[G(y)\geq G(x)+\ska{y-x,\nabla G(x)}-b\]
holds for every $y\in B_{l r}(x)$, where $\ep_0$ and $l$ are as in \autoref{lem:geo_concave}.
\end{cor}
\autoref{Lem8.4anyball} and \autoref{lem:rings} lead to the following result. The proof is obtained in the same way as in \cite[Corollary 8.5]{CaSi-09RegularityIntegroDiff}:

\begin{cor}
\label{Cor8.5}
Let $\rho_0\in(0,1)$ be arbitrary and $\ep_0$, $l$ be as in \autoref{Lem8.4anyball}. There exists a constant $C_1=C_1(n,\delta,\rho_0,\lambda)\geq 1$ and for every $x\in\Sigma$ there is radius $r\in(0,\rho_0 2^{-1/(2-\alpha)})$ and a sector, $V_\xi$, (both depending on $x$) such that 
\begin{align}
\label{8.5.1}
\frac{\bet{\{y\in R_{r}(x,V_\xi):u(y)<u(x)+\ska{y-x,\nabla\Gamma(x)}-C_1f(x)(lr)^2\}}}{\bet{R_{r}(x,V_\xi)}} \leq \ep_0
\end{align}
and
\begin{align}
\label{8.5.2}
\bet{\nabla\Gamma(B_{lr/2}(x))}\leq (8 C_1)^n f(x)^n\bet{B_{lr/2}(x)} \,,
\end{align}
where $R_{r}(x,V_\xi)$ is defined as in \autoref{Lem8.4anyball}. 
\end{cor}

\begin{proof}
Let $x\in\Sigma$ be fixed.
Because of Lemma \ref{lem:rings} there is a constant $C_0=C_0(n,\delta,\rho_0,\lambda)\geq 1$ and for every $A>0$ there exists some $r\in(0,\rho_0 2^{-1/(2-\alpha)})$ and a sector $V_\xi$ such that
\[\bet{\{y\in R_{r}(x,V_\xi):\,u(y)<u(x)+\ska{y-x,\nabla\Gamma(x)}-Ar^2\}}\leq C_0\,\tfrac{f(x)}{A}\bet{R_{r}(x,V_\xi)}.\]
By choosing $A=\frac{C_0 f(x)}{\ep_0}$ we obtain \eqref{8.5.1}, where $C_1=\frac{C_0}{\ep_0 l^2}$. 

Now let us prove \eqref{8.5.2}. First note that for every $b>0$ the set $ \{y \in \R^n:\,\Gamma(y)<\Gamma(x)+\ska{y-x,\nabla\Gamma(x)}-b\}$ is a subset of $\{y \in \R^n:\,u(y)<u(x)+\ska{y-x,\nabla\Gamma(x)}-b\}$. Using this relation and \eqref{8.5.1} we conclude that there is a constant $C_1=C_1(n,\delta,\rho_0,\lambda)\geq 1$ and some $r\in(0,\rho_0 2^{-1/(2-\alpha)})$ such that
\begin{align} \label{hilfsab8.3}
\frac{\bet{\{y\in R_{r}(x,V_\xi):\,\Gamma(y)<\Gamma(x)+\ska{y-x,\nabla\Gamma(x)}-C_1 f(x) (lr)^2\}}}{\bet{R_{r}(x,V_\xi)}}\leq \ep_0.
\end{align}
Because of the concavity of $\Gamma$ and \eqref{hilfsab8.3}, we may apply \autoref{Lem8.4anyball} for $G=\Gamma$ and $b=C_1 f(x) (lr)^2$. We obtain 
\[\Gamma(y)\geq\Gamma(x)+\ska{y-x,\nabla\Gamma(x)}-C_1f(x)(lr)^2\]
for every $y\in B_{lr}(x)$. At the same time,
\[\Gamma(y)\leq\Gamma(x)+\ska{y-x,\nabla\Gamma(x)}\]
for every $y\in B_{lr}(x)$ because of the concavity of $\Gamma$. Hence,
\begin{align}\label{impest}
\bet{\Gamma(y)-\Gamma(x)-\ska{y-x,\nabla\Gamma(x)}}\leq C_1 f(x)(lr)^2\,\,\text{for every } y\in B_{lr}(x) \,.
\end{align} 
Recall that $f$ is a positive function. \autoref{lem:techstuff_conc}(ii) -- presented below -- completes the proof.
\end{proof}

\begin{lem}\label{lem:techstuff_conc}
(i) Let $G : B_R \to \R$ be a concave function. Then
\begin{align}\label{LokLip2}
\sup_{y\in B_{R/2}}\bet{\nabla G(y)}\leq\frac{4}{R}\sup_{y\in B_R}\bet{G(y)} \,.
\end{align}
(ii) Let $G : B_R \to \R$ be a concave function satisfying 
\begin{align}\label{impest2}
\bet{G(y)- G(0)-\ska{y,\nabla G(0)}}\leq K R^2\,
\end{align} 
for every $y\in B_{R}$ with some $K > 0$. Then
\begin{align}\label{techstuff-assert}
\bet{\nabla G(B_{R/2})}\leq (8K)^n \bet{B_{R/2}} \,.
\end{align}
\end{lem}
\begin{proof}
(i) It is sufficient to prove \eqref{LokLip2} for $R=1$. Set $M=\sup_{y\in B_1}|G(y)|$. 
Let $y\in B_{1/2}$. Given $h\neq 0$, choose $s<0<t$ such that $\bet{y+sh}=\bet{y+th}=1$. Then
\[-M\leq G(y+sh)\leq G(y)+\ska{sh,\nabla G(y)}\leq M + \ska{sh,\nabla G(y)} \text{ and } \bet{s h} \geq \tfrac{1}{2} \,.\]
The same estimates hold when $s$ is replaced by $t$. Therefore we obtain
\[\ska{\nabla G(y), h}\leq -\frac{2M}{s} \leq 4M \bet{h} \;\text{ and }\; \ska{\nabla G(y), h}\geq -\frac{2M}{t}\geq- 4M \bet{h} .\]
As a consequence we deduce for every $h\neq 0$ the estimate $\frac{\bet{\ska{\nabla G(y), h}}}{\bet{h}}\leq 4M$. Hence we obtain $\bet{\nabla G(y)}\leq 4M$, which finishes the proof of \eqref{LokLip2}.

(ii) For $y\in B_{R}$ define $\widehat{G}(y)=G(y)-G(0)-\ska{y,\nabla G(0)}$. $\widehat{G}$ is a concave function in $B_{R}$. Let $z \in B_{R/2}$. Using \eqref{LokLip2} and \eqref{impest2} we obtain 
\begin{align*}
\bet{\nabla G(z)-\nabla G(0)} & = \bet{\nabla \widehat{G}(z)} \leq \frac{4}{R}\sup_{y\in B_R}
\bet{G(y)-G(0)-\ska{y,\nabla G(0)}} \leq 4 K R = 8K \tfrac{R}{2}\,.
\end{align*}
Therefore 
\[\nabla G(B_{R/2})\subset B_{8K (R/2)} (\nabla G(0)) \; \text{ and } \; \bet{\nabla G(B_{R/2})}\leq (8K)^n \bet{B_{R/2}}. \qedhere\]

\end{proof}

As a consequence of \autoref{Cor8.5} we derive a theorem which can be considered as a nonlocal finite cube substitute for the classical ABP estimate, cf. \cite[Theorem 3.2]{CaCa-95} and \autoref{cor:nonlocalABP}. 
\begin{sa}
\label{thm:cubecover}
Let $l\in(0,\tfrac12 )$ be as in \autoref{Cor8.5} and assume $0 < \rho_0 \leq \frac{l}{16n}$. There are constants $C_2=C_2(\delta,\lambda,\rho_0,n) \geq 1$ and $\nu=\nu(\delta,n)>0$ and a disjoint family of open cubes $(Q^j)_{j=1,\ldots,m}$, $m\in\N$, with diameters $0<d_j\leq \rho_02^{-1/(2-\alpha)}$ which covers the contact set $\Sigma=\{u=\Gamma\} \cap B_1$ such that the following properties hold for every $j=1,\ldots,m$:
\begin{enumerate}
\item $\displaystyle\Sigma\cap \overline{Q^j}\neq\emptyset$.
\item $\displaystyle\bet{\nabla\Gamma(\overline{Q^j})}\leq C_2(\sup_{\overline{Q^j}}f)^n\bet{Q^j}.$
\item $\displaystyle\bigl|\{y\in \eta Q^j: u(y)\geq\Gamma(y)-C_2(\sup_{\overline{Q^j}}f)d_j^2\}\bigr|\geq \nu\bet{\eta Q^j}$, where $\eta=(1+\frac{8}{l})\sqrt{n}$.
\end{enumerate} 
\end{sa}

\begin{proof}
The proof follows the one of \cite[Theorem 8.7]{CaSi-09RegularityIntegroDiff}. In our context, the main constants additionally depend on $\delta$. Let $C_1=C_1(n,\delta,\rho_0,\lambda)\geq 1$ be as in \autoref{Cor8.5}. Set $c_1 = (8 C_1)^n$ and $c_2=16 C_1$. We prove the assertion of the theorem with $C_2=c_1 \eta^n$ and $\nu=(1-l) \tfrac{\bet{R(0,V_\xi)}}{\bet{B_1}} (8\sqrt{n})^{-n}$, where $R(0,V_\xi)$ is as in \autoref{lem:geo_concave}. 

Let $\mathcal{Q}_1$ be a finite disjoint family of open cubes $Q$ with diameter $d_1 = \rho_0 2^{-1/(2-\alpha)}$ and the property $B_1 \subset \bigcup\limits_{\mathcal{Q}_1} \overline{Q}$. Let $\mathcal{Q}'_1 \subset \mathcal{Q}_1$ be the subfamily of all cubes $Q$ with $\overline{Q} \cap \Sigma \ne \emptyset$. We decompose every cube in $\mathcal{Q}'_1$ which does not satisfy both conditions (2) and (3) from above into $2^n$ sub-cubes with half diameter. Now, let $\mathcal{Q}_2$ be the family of these newly created sub-cubes plus those cubes from $\mathcal{Q}'_1$ that do satisfy both conditions (2) and (3) from above (and hence were not decomposed). We repeat this procedure and obtain a sequence of families 
\[ \mathcal{Q}_1, \mathcal{Q}_2, \mathcal{Q}_3, \ldots \]
We claim that there is an index $k \in \N$ with $\mathcal{Q}_k = \mathcal{Q}_{k+i}$ for all $i \in \N$. In this case, we set $m = \# \mathcal{Q}_k$. 

Let us assume that no such index $k \in \N$ exists. Then there exists a sequence of cubes $Q^j$ with diameter $d_j$ such that $d_j = 2^{-j+1} d_1$ and for every $j\in\N$ the following conditions hold:
\begin{itemize}
 \item $\displaystyle Q^j\supset Q^{j+1}$.
\item $\displaystyle\overline{Q^j}\cap\Sigma\neq\emptyset$.
\item $Q^j$ violates $(2)$ or $(3)$.
\end{itemize}
Let $x_0 \in \R^n$ satisfy $\{x_0\} = \bigcap\limits_{j \in \N} \overline{Q^j}$. Firstly, we claim $x_0 \in \Sigma=\{u=\Gamma\} \cap B_1$. It is sufficient to prove $x_0\in \{u=\Gamma\}$. Note that there is a sequence $(x_j)_{j\in\N}$ with $x_0 = \lim\limits_{j \to \infty} x_j$ and $x_j \in\overline{Q^j}\cap\Sigma$ for every $j \in \N$, and hence $x_0\in\Sigma$ since $\Sigma$ is closed.

We now derive a contradiction by showing that one of the cubes $Q^j$ from above satisfies $(2)$ and $(3)$. Using \autoref{Cor8.5}, there is a number $r$ with $0<r<\rho_0 2^{-1/(2-\alpha)}$ such that
\begin{align}\label{Use1Cor8.5}\frac{\bet{\{y\in R(x_0,V_\xi):u(y)<u(x_0)+\ska{y-x_0,\nabla\Gamma(x_0)}-C_1f(x_0)(lr)^2\}}}{\bet{R(x_0,V_\xi)}}\leq l\end{align}  
and
\begin{align}\label{Use2Cor8.6}\bet{\nabla\Gamma(B_{lr/2}(x_0))}\leq c_1f(x_0)^n\bet{B_{lr/2}(x_0)}.\end{align}
Fix an index $j_0 \in \N$ such that $\frac{lr}{4}\leq d_{j_0}< \frac{lr}{2}$. Therefore
\begin{align}\label{Verschachtelung}B_{lr/2}(x_0)\supset\overline{Q^{j_0}}, \quad B_{r}(x_0)\subset\eta Q^{j_0}\subset B_3.\end{align}

Note that $\Gamma(y)\leq u(x_0)+\ska{y-x_0,\nabla\Gamma(x_0)}$ in $B_3$. Recall $\eta=(1+\frac{8}{l})\sqrt{n}$. Using \eqref{Use1Cor8.5}, \eqref{Verschachtelung} and the relation between $d_{j_0}$ and $r$, we obtain
\allowdisplaybreaks
\begin{align*}
&\bigl|\{y\in \eta Q^{j_0}: u(y)\geq\Gamma(y)-C_2(\sup\nolimits_{\,\overline{Q^{j_0}}}f)d_{j_0}^2\}\bigr| \\
&\quad\geq \bigl|\{y\in \eta Q^{j_0}: u(y)\geq u(x_0)+\ska{y-x_0,\nabla\Gamma(x_0)}-c_2 f(x_0)\tfrac{(lr)^2}{16}\}\bigr|\\*[0.1cm]
&\quad\geq\bigl|\{y\in R(x_0,V_\xi): u(y)\geq u(x_0)+\ska{y-x_0,\nabla\Gamma(x_0)}-C_1 f(x_0)(lr)^2\}\bigr|\\
&\quad\geq\bet{R(x_0,V_\xi)}-l\bet{R(x_0,V_\xi)}=(1-l)\bet{R(x_0,V_\xi)}\geq\nu\bet{\eta Q^{j_0}}.
\end{align*}
Moreover, using \eqref{Use2Cor8.6} and \eqref{Verschachtelung}, we obtain
\begin{align*}
\bet{\nabla\Gamma(\overline{Q^{j_0}})}&\leq\bet{\nabla\Gamma(B_{lr/2}(x_0))}\leq c_1 f(x_0)^n\bet{B_{lr/2}(x_0)} \leq c_1(\sup_{\overline{Q^{j_0}}}f)^n\bet{\eta Q^{j_0}}=C_2(\sup_{\overline{Q^{j_0}}}f)^n\bet{Q^{j_0}}.
\end{align*}
Therefore $Q^{j_0}$ satisfies $(1)-(3)$ with $C_2,\nu$ from above. Contradiction.
\end{proof}

The following corollary can be seen as a discretized version of the Aleksandrov-Bakelman-Pucci estimate \cite[Theorem 3.2]{CaCa-95} in our setting. Note that the index $m$ in the assertion below depends on $\alpha$ with $m \to +\infty$ for $\alpha \to 2-$.

\begin{cor}
\label{cor:nonlocalABP}
Under the assumptions of \autoref{thm:cubecover} we have
\[\sup_{B_1}u^+\leq C_3\biggl(\sum_{j=1}^m(\sup_{\overline{Q^j}}f)^n\bet{Q^j}\biggr)^{1/n},\]
with $m \in \N$, $(Q^j)$ as in \autoref{thm:cubecover} and $C_3=C_3(n,\delta,\rho_0,\lambda)\geq 1$.
\end{cor}
\begin{proof}
Set $S=\sup_{B_1}u^+$. Since $u^+=0$ in $\R^n\setminus B_{1}$ and $u$ is upper semicontinuous, there is $x_0 \in B_1$ with $S=u^+(x_0)$. Using the geometric argument given in the proof of \cite[Lemma 9.2]{GiTr-98} we deduce
\[B_{S/4}\subset\nabla\Gamma(B_1) \quad \Rightarrow  \quad  S^n\leq c_1 \bet{\nabla\Gamma(B_1)} \,,\]
with some constant $c_1=c_1(n) \geq 1$. Part (2) of \autoref{thm:cubecover} now implies 
\[\sup_{B_1}u^+\leq c_1^{1/n} \bet{\nabla\Gamma(B_1)}^{1/n}= c_1^{1/n} \bet{\nabla\Gamma(B_1\cap\{u=\Gamma\})}^{1/n}\leq C_3\biggl(\sum_{j=1}^m(\sup_{\overline{Q^j}}f)^n\bet{Q^j}\biggr)^{1/n},\]
where $C_3\geq 1$ depends only on $n$, $\delta$, $\rho_0$ and $\lambda$. Here, we have used the fact that
\begin{align}\label{factconcaveenvelope} 
\nabla\Gamma(B_1\cap\{u=\Gamma\})=\nabla\Gamma(B_1)\,,
\end{align}
which follows from (\ref{eq:conc_env}) and the properties $u\leq0$ in $\R^n\setminus B_1$ and $u^+ \not\equiv 0$.
\end{proof}

 
\section{A Special Bump Function}\label{sec:SpecialFunc}

In this section we construct a special function with the properties as the one in \cite[Lemma 4.1]{CaCa-95}. We will use this function in Section \ref{sec:PointEsti} in combination with the ABP substitute from the previous section. The construction is based on an idea used in \cite{CaSi-09RegularityIntegroDiff}, but differs significantly to deal with the fact that the mass of the kernels, $K(y)$, could be concentrated on only a small sector, $V_\xi$.  This special function will appear at the end of this section in Corollary \ref{SpecialCor:special_func}.

To begin, we will consider a two parameter family of functions $f_{\gam,p}\in C^{1,1}(\real^n)$ given by

\begin{equation*}
	f_{\gam,p}(y) = \hat f(\abs{y})
\end{equation*}
and
\begin{equation}\label{SpecialEq:fHatDef}
	\hat f(r)=
	\begin{cases}
		r^{-p}\ &\text{if}\ r\geq 1-\frac{C_4}{2}\\
		m_{\gam,p}(r)\ &\text{if}\ 1-C_4\leq r \leq 1-\frac{C_4}{2}\\
		\gam^{-p}\ &\text{if}\ r\leq 1-C_4.
	\end{cases}
\end{equation}
Here we choose to take the middle function, $m_{\gam,p}$, so that $\hat f$ is $C^{1,1}(\real^n)$ and monotone decreasing for $r\in[0,\infty)$.  The value of $C_4\in(0,1)$ depends only on $\del$ via the opening of the sectors, $V_\xi$, and is chosen so that for some universal $\mu_0>0$

\begin{equation}\label{SpecialFunEq:ChoiceC0}
\text{for all}\ \xi\in\sphere^{n-1},\	\abs{\{y\in V_\xi\ :\ e_1 + y\in B_{1-C_4}(0)\}}\geq \mu_0.
\end{equation}

The special function will be constructed in two phases.  First in Lemma \ref{SpecialLem:PLarge} we will find a value of $p$ large enough that we can make $M^-_{\A_{sec}}f_{\gam,p}\geq 0$ for all $\al$ near $2$.  Then in Lemma \ref{SpecialLem:GammaSmall} we take $\gam$ small enough to cover the range of $\al$ down to $\al_0$.  Before we get to those results,
we note a couple of useful properties of the family $\{f_{\gam,p}\}$.

\begin{note}\label{SpecialNote:GammaOrdering}
	If $\gam_1<\gam_2$ and $p$ is fixed, then for all $y$
	\begin{equation*}
		f_{\gam_1,p}(y)\geq f_{\gam_2,p}(y),
	\end{equation*}
	and the two functions are equal when $\displaystyle \abs{y}\geq 1-\frac{C_4}{2}$, hence 
	\begin{equation*}
		M^{-}_{\A_{sec}}f_{\gam_1,p}(x)\geq M^-_{\A_{sec}}f_{\gam_2,p}(x),
	\end{equation*}
	for all $\displaystyle \abs{x}\geq 1-\frac{C_4}{2}$.
\end{note}

We also record a useful reduction for the computations. 

\begin{lem}[{\cite[p.623]{CaSi-09RegularityIntegroDiff}}]\label{SpecialLem:ReductionToe1}
	Let $f=f_{\gam,p}$. If $M^-_{\A_{sec}}f(e_1)\geq 0$, then $M^-_{\A_{sec}}f(x)\geq 0$ for all $\abs{x}\geq 1$.
\end{lem}

\begin{proof}[Proof of Lemma \ref{SpecialLem:ReductionToe1}]
	First, we note that $M^-_{\A_{sec}}$ is rotationally invariant due to the definition of $\A_{sec}$ (it could fail to be rotationally invariant if $\A_{sec}$ was a smaller collection of kernels).  Therefore by the radial symmetry of $f$ we see that
	\begin{equation*}
		M^-_{\A_{sec}}f(x) = M^-_{\A_{sec}}f(\abs{x}e_1)
	\end{equation*} 
	for all $x\in\real^n$ for which $M^-_{\A_{sec}}$ is well defined. 
	
	Second, we use \autoref{SpecialNote:GammaOrdering} to reduce the calculation to the lowest value of $\gam$, say $\gam_0$ which will be fixed below (in fact $\gam_0=1$ will suffice).  Indeed, assuming we have proved that 
	\begin{equation*}
		M^-_{\A_{sec}}f_{\gam_0,p}(x)\geq 0,
	\end{equation*} 
	then note \ref{SpecialNote:GammaOrdering} gives that for all $\gam<\gam_0$
	\begin{equation*}
		M^-_{\A_{sec}}f_{\gam,p}(x)\geq M^-_{\A_{sec}}f_{\gam_0,p}(x).
	\end{equation*}

	Third we see that if
	\begin{equation*}
		\tilde f(x) = c^pf(cx),
	\end{equation*} 
	then whenever $c>1$ and $\displaystyle\abs{x}>1-\frac{C_4}{2}$
	\begin{equation*}
		\tilde f(x) = c^p\abs{cx}^{-p} = f(x),
	\end{equation*}
	and one can check that $\tilde f(x)\geq f(x)$ for all $x$.
	
	To conclude, let $x_0=\abs{x_0}e_1$ be fixed, and let $c=\abs{x_0}$.  Then we note that (using $f(x)=c^{-p}\tilde f(\frac{x}{c})$)
	\allowdisplaybreaks
	\begin{align*}
		M^-_{\A_{sec}}f_{\gam_0,p}(x_0) &= M^-_{\A_{sec}}c^{-p}\tilde f(\frac{\cdot}{c})(x_0)\\
		&= c^{-p}c^{-\al}M^-_{\A_{sec}}\tilde f(\frac{x_0}{c})\\
		&\geq c^{-p-\al}M^-_{\A_{sec}}f(\frac{x_0}{c})\\
		&=c^{-p-\al}M^-_{\A_{sec}}f(e_1)\\
		&\geq 0 \ \text{(by assumption)}.
	\end{align*}
	 
\end{proof}

\begin{lem}\label{SpecialLem:PLarge}
	Let $\gam_0=1$ be fixed.  Then, there exists a $p_0$ and an $\al_1$, depending only on $\gam_0$, $C_0$, $\del$, $n$, $\lam$, $\Lam$,   such that 
	\begin{equation*}
		M^-_{\A_{sec}}f_{\gam_0,p}(x)\geq 0 \ \ \ \text{for all}\ \ \abs{x}>1,
	\end{equation*}
	 for all orders, $\al\in(\al_1,2)$.
\end{lem}

\begin{proof}[Proof of Lemma \ref{SpecialLem:PLarge}]
	We first note that by Lemma \ref{SpecialLem:ReductionToe1}, it suffices to estimate only $M^-_{\A_{sec}}f(e_1)$.
	
	Let $K$ be any kernel in $\A_{sec}$.  In the end of the proof, no constants will depend on this particular $K$.  Let $\Indicator_\xi$ be its corresponding lower bound sector.
	Let us drop the parameters $\gam,p$ from $f$ for ease of notation.  We will first include some preliminary calculations and choices of constants.  Then at the end, we put all of these calculations together to conclude the lemma.
	
	For this part of the construction of the special function, $\al$ will be close to $2$, and therefore it is the local behavior of $f$ which is essential.  We start by focusing on the contribution to $M^-_{\A_{sec}}f(e_1)$ given by the integration for $h\in B_r$.  We make an important emphasis that for each $\gam,p, \al$ there is a direction which optimizes $M^-_{\A_{sec}}f(e_1)$ which depends on all of $\gam$, $p$, and $\al$ via (\ref{PointEq:FormulaMPlus}).  Let us call that direction $\xi$ throughout, but with the understanding that it depends on $\gam,p,\al$.  None of the estimates we prove will depend at all on the specific choice of $\xi$-- they will only depend on the opening of the sector, $\del$, and the other universal parameters.
		
		Whenever $\displaystyle r< \frac{C_4}{2}$, then the inequality (\cite[p.624]{CaSi-09RegularityIntegroDiff}) holds
		\begin{equation}\label{SpecialFunEq:DelfEst}
			\del^2_h f(e_1) \geq p\left((-\abs{h}^{2}+(p+2)(h_1)^2-\frac{1}{2}(p+2)(p+4)(h_1)^2\abs{h}^2)\right).
		\end{equation}
	Therefore we fix let $r_0=\frac{C_4}{2}$ for the remaining calculations.  Next we observe
		if $h\in\real^n\setminus B_{r_0}$, then 
		\begin{equation}\label{SpecialFunEq:DelfEstFarAway}
			\del^2_h f(e_1)\geq 2(\inf(f)-f(e_1))=-2.
		\end{equation}
		We can thus conclude a bound from below on the contribution outside of $B_{r_0}$,
		\begin{align}
			(2-\al)\int_{\real^n\setminus B_{r_0}} -\Lam(\del^2_h f(e_1))^++\lam\Indicator_{V_\xi}(\del^2_h f(e_1))^-\mu(dh) &\geq (2-\al)\int_{\real^\setminus B_{r_0}}-\Lam(\del^2_h f(e_1))^-\mu(dh)\nonumber\\
			&\geq (-2)(2-\al)\int_{\real^n\setminus B_{r_0}}\Lam\mu(dh).\label{SpecialNote:DelUMinusControl}
		\end{align}

		In order to determine the good value of $p_0$, we note first there exists a $\mu_1$ which depends only on $n$ and the sector opening, $\del$, such that for all $\xi$
		\begin{equation}\label{SpecialFunEq:Mu1Def}
			\int_{\partial B_1}(z_1)^2\Indicator_{V_\xi}(z)dS(z)\geq \mu_1.
		\end{equation}
		Choose $p_0$ large enough so that 
		\begin{equation}\label{SpecialFunEq:ChoiceOfC2}
			\lam(p_0+2)\int_{\partial B_1}(z_1)^2\Indicator_{V_\xi}(z)dS(z)-\Lam\int_{\partial B_1}\abs{z}^2dS(z)\geq C_5 >0.
		\end{equation}
		We point out that this choice of $p_0$ (and hence $C_5$) depends only on $\mu_1$, $n$, $\lam$, $\Lam$, and since $\mu_1$ depends only on $n$ and $\del$, $p_0$ depends on those as well.  It is essential to note that $p_0$ does not depend on $\gam_0$.

	Now we can estimate the contribution to $M^-_{\A_{sec}}f(e_1)$ from $B_{r_0}$.  We note that we only care about $(\del^2_y f(e_1))^+$.  This is one place where the original definition of $M^-_{\A_{sec}}$ in (\ref{PrelimNotationEq:MMinusDef}) is more helpful than the formula from Lemma \ref{PointEq:FormulaMPlus}.  First we use (\ref{SpecialFunEq:DelfEst}) then (\ref{specialK}), and finally (\ref{SpecialFunEq:ChoiceOfC2}).
	
	\begin{align}
		&\int_{B_{r_0}} (\del^2_h f(e_1)) K(h)dh = \int_0^{r_0}s^{n-1}\int_{\partial B_1} (\del^2_{sz} f(e_1))K(sz)dS(z) ds\nonumber\\		
		&\ \ \ \ \geq \int_0^{r_0}s^{n-1}\int_{\partial B_1} p_0\left((-\abs{sz}^{2}+(p_0+2)(sz_1)^2-\frac{1}{2}(p_0+2)(p_0+4)(sz_1)^2\abs{sz}^2)\right) K(sz)dS(z)ds\nonumber\\
		&\ \ \ \ \geq (2-\al)\int_0^{r_0}s^2s^{-1-\al}\bigg[\int_{\partial B_1}\lam p_0(p_0+2)(z_1)^2\Indicator_{V_\xi}(z)dS(z) - \int_{\partial B_1} \Lam p_0\abs{z}^{2}dS(z)\nonumber\\ 
		&\ \ \ \ \ \ \ \ \ \ \ \ \ \ \ \ \  \ - s^2\int_{B_{r_0}}\Lam\frac{1}{2}p_0(p_0+2)(p_0+4)(z_1)^2\abs{z}^2 dS(z)\bigg]ds\nonumber\\
		&\ \ \ \ \geq (2-\al)\int_0^{r_0}s^2s^{-1-\al} [p_0C_5  -s^2\om(n)\Lam\frac{1}{2}p_0(p_0+2)(p_0+4)]ds\nonumber \\
		&\ \ \ \ =\lam p_0 C_5(r_0)^{2-\al} -   \om(n)\Lam\frac{1}{2}\frac{2-\al}{4-\al}p_0(p_0+2)(p_0+4)(r_0)^{4-\al}.\label{SpecialFunEq:EstInBr}
	\end{align}
	
	And we put all the pieces together with (\ref{SpecialFunEq:EstInBr}) and (\ref{SpecialFunEq:DelfEstFarAway})
 	\begin{align}
		&L_K(f,e_1) = (2-\al)\int_{\real^n}\del^2_h f(e_1)K(h)dh\nonumber\\
		&\ \ \ \ = \int_{B_{r_0}} \del^2_h f(e_1)K(h)dh + \int_{\real^n\setminus B_{r_0}}\del^2_h f(e_1)K(h)dh\nonumber\\
		&\ \ \ \ \geq \lam p_0 C_5(r_0)^{2-\al} -   \om(n)\Lam\frac{1}{2}\frac{2-\al}{4-\al}p_0(p_0+2)(p_0+4)(r_0)^{4-\al}
		-2(2-\al)\int_{\real^n\setminus B_{r_0}}\Lam\mu(dh)\label{SpecialFunEq:LargePConcludeLine}
	\end{align}
	
	To conclude, we take $\al_1$ close enough to $2$ so that for all $\al\in(\al_1,2)$, (\ref{SpecialFunEq:LargePConcludeLine}) is $\geq 0$.  Hence by the definition of $M^-_{\A_{sec}}$, it follows that $M^-_{\A_{sec}}f(e_1)\geq0$, which concludes the lemma.	
\end{proof}

Now that we have the behavior of the special function controlled for $\al\in(\al_1,2)$, we need to get the behavior for $\al\in(\al_0,\al_1]$.

\begin{lem}\label{SpecialLem:GammaSmall}
	Let $\al_0$ be the lower bound on $\al$ given in the introduction.  Let $\al_1$ and $p_0$ be fixed from Lemma \ref{SpecialLem:PLarge}.  Then there exists $\gam_1>0$, depending only on $C_4$, $\mu_0$ (and hence $\del$) $\al_0$, $\al_1$, $p_0$, $n$, $\lam$, $\Lam$, such that for all $\al\in(\al_0,2)$ and $\abs{x}>1$
	\begin{equation}
		M^-_{\A_{sec}}f_{\gam_1,p_0}(x)\geq 0.
	\end{equation}
\end{lem}

\begin{proof}[Proof of Lemma \ref{SpecialLem:GammaSmall}]
Just to be concrete, let $\gam_0=1/4$.  There are some lower bounds which will be much easier to treat if we have one function to plug into the integrals for all choices of $\gam$.  To this end, we introduce an auxiliary function simply for the sake of some estimates.  Let $\phi_{\gam,p}$ be 
	\begin{equation}
		\phi_{\gam,p}(y) = \min\{\gam^{-p},\abs{y}^{-p}\}.
	\end{equation}
	We know that $\phi_{\gam,p}\in C^2(\real^n\setminus B_\gam)$.

		Furthermore, we can use the function $\phi_{\gam,p_0}$ as a lower bound for $f_{\gam,p_0}$.  Indeed,  let $\gam\leq\gam_0$.  Then we have that $\phi_{\gam,p_0}(y)=f_{\gam,p_0}(y)$ whenever $\displaystyle \abs{y}\geq 1-\frac{C_4}{2}$, and also that $\phi_{\gam,p_0}(y)\leq f_{\gam,p_0}(y)$ for all $y$.  Also for $\abs{x}\geq 1$,
		\begin{equation}
			M^-\phi_{\gam,p_0}(x)\geq -C\norm{\phi_{\gam,p_0}}_{C^{1,1}(\real^n\setminus B_{1/2})}.
		\end{equation} 
	Finally we note that the $C^{1,1}(\real^n\setminus B_{1/2})$ norm of $\phi_{\gam,p_0}$ is independent of $\gam$.

		   Using $\phi_{\gam_0,p_0}$ we see that there exists a $C$ depending only on $n$, $\lam$, $\Lam$ such that
		\begin{equation}
			\Lam\int_{\real^n}-(\del f_{\gam,p_0}(e_1,y))^-\abs{y}^{-n-\al}dy\geq -C(\norm{\phi_{\gam_0,p_0}}_{C^{1,1}(B_{1/2}(e_1))} + 2(1-\inf_{\real^n}(f))),
		\end{equation}
for all $\gam\leq\gam_0$.
	
We now proceed with the calculation which will lead to the choice of $\gam_1$.  Just as in the previous lemma, we work with $L_K$ instead of $M^-$, and the result follows because no estimates depend on the particular choice of $K$.
	
	We drop the $\gam,p_0$ subscripts until the end.
	\begin{align}
		&\int_{\real^n}\del^2_h f(e_1)K(h)dh\nonumber\\
		&\ \ \ \ \geq (2-\al)\bigg[\lam\int_{\real^n}(\del^2_h f(e_1))^+\Indicator_{V_\xi}(h)\abs{h}^{-n-\al}dh
		 -\Lam\int_{\real^n}(\del^2_h f(e_1))^-\abs{h}^{-n-\al}\bigg]\nonumber\\
		&\ \ \ \ \geq (2-\al)\lam\int_{\{y\ :\ e_1\pm y \in B_{1-C_4}(0)\}  } (\del^2_h f(e_1))^+\Indicator_{V_\xi}(h)\abs{h}^{-n-\al}dh\nonumber\\
		&\ \ \ \ \ \ \ \ \ \ \ \  - (2-\al)\Lam C(n,\al_0) (\norm{\phi_{\gam_0,p_0}}_{C^{1,1}( B_{1/2}(e_1))} + 2(1-\inf_{\real^n}(f)))\nonumber\\
		&\ \ \ \ \geq (2-\al)\lam\int_{\{y\ :\ e_1\pm y \in B_{1-C_4}(0)\}  } (\gam^{-p_0}-1)\abs{h}^{-n-\al}dh\nonumber\\
		&\ \ \ \ \ \ \ \ \ \ \ \  - (2-\al)\Lam C(n,\al_0) (\norm{\phi_{\gam_0,p_0}}_{C^{1,1}( B_{1/2}(e_1))} + 2\nonumber\\
		&\ \ \ \ \geq (2-\al)\lam (2-C_4)^{-n-\al}(\gam^{-p_0}-1)\abs{\{y\ :\ e_1\pm y \in B_{1-C_4}(0)\}  }\nonumber \\
		&\ \ \ \ \ \ \ \ \ \ \ \  - (2-\al)\Lam C(n,\al_0) (\norm{\phi_{\gam_0,p_0}}_{C^{1,1}( B_{1/2}(e_1))} + 2)\nonumber\\
		&\ \ \ \ \geq (2-\al)\lam (2-C_4)^{-n-\al}(\gam^{-p_0}-1)\mu_0 - (2-\al)\Lam C(n,\al_0) (\norm{\phi_{\gam_0,p_0}}_{C^{1,1}( B_{1/2}(e_1))} + 2)\nonumber
	\end{align}
	
	\noindent 
	We note that if $h\in\{y\ :\ e_1\pm y \in B_{1-C_4}(0)\}$, then $C_4\leq \abs{h}\leq 2-C_4$, and hence $\abs{h}^{-n-\al}\geq (2-C_4)^{-n-\al}$.  Also we note $\inf(f)=0$.  Now $\gam_1$ can be chosen to depend on only $\al_1$, $p_0$, $\mu_0$ (and hence $\del$), $n$, $\lam$, $\Lam$, so that the final line becomes $\geq 0$.
\end{proof}

\begin{cor} \label{SpecialCor:special_func}
Assume $0 < R \leq 1$ is given. There exists a continuous function $\Phi:\R^n\rightarrow\R$ (depending on $R$) with the following properties:
\begin{enumerate}
\item  $\Phi(x)=0$ for every $x\in \R^n\setminus B_{2\sqrt{n}}$.
\item $\Phi(x)>2$ for every $x\in Q_3$.
\item There exists a bounded, nonnegative function $\psi:\R^n\to\R$ (depending on $R$), supported in $\overline{B_R}$, such that $M^-_{\A_{sec}}\Phi(x)\geq -\psi(x)$ for every $x\in\R^n$ and every $\alpha\in(\alpha_0,2)$.
\end{enumerate}
\end{cor}

\begin{proof}[Proof of Corollary \ref{SpecialCor:special_func}]
We begin by noting that by construction, $f_{\gam_0,p_0}$ is $C^{1,1}(\real^n)$.  We therefore have that given any $0<R<1$ fixed, the function $\displaystyle f_R(x):= f(\frac{x}{R})$ satisfies classically by a rescaling of $M^-_{\A_{sec}}f(x)\geq0$ in $\{\abs{x}\geq1\}$,
	\begin{equation}
		M^-_{\A_{sec}}f_R(x)=R^{-\al}M^-_{\A_{sec}}f(\frac{x}{R})\geq0\ \ \text{whenever}\ \ \abs{x}\geq R.
	\end{equation}

Next we can subtract the constant $(2\sqrt{n})^{-p_0}$. Since constants are subsolutions of $M^-_{\A_{sec}}\geq0$ and the maximum of two subsolutions is still a subsolution, we see that
\begin{equation}
	\tilde f_R = \max\{f_R-(2\sqrt{n})^{-p_0},0\}
\end{equation}	
	still satisfies $M^-_{\A_{sec}}\tilde f_R\geq0$ in the set $\abs{x}\geq R$.
	
	Finally to conclude, we choose $c$ large enough so that
	\begin{equation}
		\Phi(x):= c \max\{f_R(x)-(2\sqrt{n})^{-p_0},0\}>2 \ \ \text{for all}\ x\in Q_{3}.
	\end{equation}
	To conclude we comment that by \autoref{lem:ContinuityOfF}, $M^-_{\A_{sec}}\Phi$ is continuous in $\real^n$, and hence $\psi$ is continuous.
\end{proof}


\section{Point To Measure Estimates and H\"older Regularity} \label{sec:PointEsti}

This section contains the main auxiliary result, \autoref{lem:point_estimate}, which is the key to \autoref{thm:holder}. The proof of \autoref{lem:point_estimate} uses the main contributions of this article, \autoref{thm:cubecover} and \autoref{SpecialCor:special_func}.  Once \autoref{lem:point_estimate} is established, a-priori Hölder regularity estimates follow by the classical method of oscillation reduction. 

\begin{lem}\label{lem:point_estimate}
There exist constants $\varepsilon_0>0$, $\kappa\in(0,1)$ and $A>1$ (depending only on $\lambda$, $\Lambda$, $n$, $\delta$ and $\alpha_0$) such that for every $\alpha\in(\alpha_0,2)$ and every bounded function $w:\R^n\to\R$ which is lower semicontinuous in $\overline{Q_{4\sqrt{n}}}$ and satisfies
\begin{enumerate}
 \item $w\geq 0$ in $\R^n$ ,
\item $\inf\limits_{Q_3}w\leq 1$, and
\item $M^-_{\A_{sec}}w\leq\varepsilon_0$ in $Q_{4\sqrt{n}}$ in the viscosity sense,
\end{enumerate}
we have \begin{align}\label{blabla}\bet{\{w\leq A\}\cap Q_1}\geq\kappa.\end{align}
\end{lem}

\begin{proof}
The proof uses the same strategy as the one of \cite[Lemma 10.1]{CaSi-09RegularityIntegroDiff}. Here, the size of the cubes used in the localization argument depends on $\delta$. Let $l\in(0,\tfrac{1}{2})$ be as in \autoref{Lem8.4anyball}. Set $R=\frac{l}{8}$ and $u=\Phi-w$, where $\Phi$ is the special function constructed in \autoref{SpecialCor:special_func}. Let us summarize properties of $u$. 
\begin{itemize}
 \item $u$ is upper semicontinuous in $\overline{B_{2\sqrt{n}}}$\,,
 \item $u\leq 0$ in $\R^n\setminus B_{2\sqrt{n}}$\,,
 \item For every $\alpha \in (\alpha_0,2)$ $M_\alpha^+u\geq M_\alpha^-\Phi-M_\alpha^-w\geq -\psi-\varepsilon_0$ in $Q_{4\sqrt{n}}\supset B_{2\sqrt{n}}$ in the viscosity sense, where $\psi:\R^n\to\R$ is as in \autoref{SpecialCor:special_func}.
\end{itemize}
Let $\Gamma$ be the concave envelope of $u^+$ in $B_{6\sqrt{n}}$. Next, we apply rescaled versions of \autoref{thm:cubecover} and \autoref{cor:nonlocalABP} with $\rho_0 =2\sqrt{n} \frac{l}{16 n} =\frac{l}{8\sqrt{n}}$. Note that \autoref{thm:cubecover} is formulated for subsolutions in $B_1$. Using a scaling argument, the assertion remains true  when considering subsolutions in $B_{2\sqrt{n}}$ as here. Let  $(Q^j)_{j=1,\ldots,m}$ be the family of cubes in the rescaled version of \autoref{thm:cubecover}, i.e. with diameters $d_j \leq \frac{l}{8\sqrt{n}} 2^{-1/(2-\alpha)}$. From \autoref{cor:nonlocalABP} we conclude 
\begin{align*}\sup_{B_{2\sqrt{n}}}u\leq C_3\biggl(\sum_{j=1}^m(\sup_{\overline{Q^j}}\psi+\varepsilon_0)^n\bet{Q^j}\biggr)^{1/n}\leq c_1\varepsilon_0+c_1\biggl(\sum_{j=1}^m(\sup_{\overline{Q^j}}\psi)^n\bet{Q^j}\biggr)^{1/n},
\end{align*}
where $c_1=c_1(\lambda,\delta,n) \geq C_3$. The properties $\inf\limits_{Q_3} w \leq 1$ and $\Phi> 2$ in $Q_3$  imply $\sup\limits_{B_{2\sqrt{n}}} u \geq 1$. Set $\varepsilon_0 = \frac{1}{2c_1}$. Since $\psi$ is supported in $\overline{B_R}$, we obtain
\[\frac{1}{2}\leq c_1 \sup_{B_R}\bet{\psi} \biggl(\sum_{\substack{j=1,\ldots,m\\Q^j\, \cap \,B_{R}\neq\emptyset}}\bet{Q^j}\biggr)^{1/n}.\]
Hence there is $c_2 = c_2(\delta,\lambda,\Lambda,\alpha_0,n) >0$ such that 
\begin{align}
\label{boundbelowc}
\sum_{\substack{j=1,\ldots,m\\ Q^j\, \cap \,B_{R}\neq\emptyset}}\bet{Q^j}\geq c_2.
\end{align}
Set $c_3 =C_2\bigl(\sup_{\R^n}(\psi+\varepsilon_0)\bigr)$ with $C_2$ from \autoref{thm:cubecover} which we now apply: There is $\nu>0$ such that for every $j=1,\ldots, m$
\begin{align} \label{Anotherbound2}
\bigl|\{y\in\eta Q^j : \, u(y)\geq \Gamma(y)-c_3 d_j^2\}\bigr| \geq \nu\bet{\eta Q^j}\,.
\end{align}

Let us consider the family 
\[\mathcal{Q}=\{\eta Q^j : Q^j\cap B_{R}\neq\emptyset,\, j\in\{1,\ldots,m\}\},\]
which is an open covering of the set \[ U =\bigcup_{\substack{j=1,\ldots,m \\Q^j\cap B_{R}\neq\emptyset}} \overline{Q^j}.\]

Given a finite family $\mathcal{A}= \{ A_i : 1 \leq i \leq N \}$ of sets $A_i \subset \R^n$ we define the overlapping number by $\omega(\mathcal{A}) = \max\limits_{1 \leq i \leq N} \# \{A\in \mathcal{A} : A \cap A_i \ne \emptyset\}$. In general, $\omega(\mathcal{Q})$ depends on $m$. However, by a simple covering argument there is a subfamily $\mathcal{Q}'$ of $\mathcal{Q}$ that still covers $U$ but with $\omega(\mathcal{Q}')$ independent of $m$. 

Note that the diameters $d_j$ of the cubes $Q^j$ satisfy $d_j \leq \frac{l}{8\sqrt{2n}}$. Furthermore, $Q^j \cap B_R \ne \emptyset$ implies $\eta Q^j \subset B_{1/2}$ due to the choice of $\rho_0$ and because of $\eta=(1+\frac{8}{l})\sqrt{n}$. Thus it follows from \eqref{boundbelowc} and \eqref{Anotherbound2} that
\begin{align}\label{entschung}
\bet{\{y\in B_{1/2}: \, u(y)\geq \Gamma(y)-c_3 \rho_0^2\}}\geq \kappa,
\end{align}

where $\kappa\in(0,1)$ only depends on $\Lambda,\lambda, \delta,\alpha_0$ and $n$.  Let $A_0=\sup_{B_{1/2}}\Phi$. Since for $y \in B_{1/2}$ $u(y)\geq \Gamma(y)-c_3 \rho_0^2$ implies $w(y)\leq A_0+c_3 \rho_0^2$, we obtain from  \eqref{entschung} 
\[\bet{\{y\in B_{1/2} : \, w(y)\leq A_0+c_3 \rho_0^2\}}\geq \kappa.\]
Set $A=A_0+c_3 \rho_0^2$. Since $B_{1/2}\subset Q_1$, we have
\[\bet{\{y\in Q_1 : \, w(y)\leq A\}}\geq \kappa,\]
which finishes the proof.  
\end{proof}

The typical presentation of H\"older regularity from the point-to-measure estimates proceeds via the oscillation reduction lemma of De Giorgi, which actually uses the Harnack inequality (see \cite[Chapter 4]{CaCa-95}, \cite[Chapters 8, 9]{GiTr-98}).  However as mentioned in Section \ref{sec:Background}, the Harnack inequality fails for operators in the class $\A_{sec}$, and so some care must be used.  That is why we choose to cite the direct methods presented in \cite[Section 12]{CaSi-09RegularityIntegroDiff}.  There it is directly shown how \autoref{lem:point_estimate} implies \autoref{thm:holder}.  Furthermore once \autoref{lem:point_estimate} is established, the particular properties of $\A_{sec}$ are no longer relevant.  Thus we conclude \autoref{thm:holder}.


\section{Appendix} \label{sec:Appendix}

\subsection{Failure of Harnack Inequality}\label{sec:HarnackFail}

We include here some of the details of the counterexample of the Harnack inequality from \cite{BoSz05}.

\begin{ex}[{\cite[p.148]{BoSz05}}]\label{AppendixExample:HarnackFail}
For $k \in \N$ let $I_k$ be any set of the form 
\[
I_k = \big( B_{4^{-k}}(\xi_k) \cup B_{4^{-k}}(-\xi_k) \big) \cap S^{n-1},
\]
 where $\xi_k \in \sphere^{n-1}$ are chosen such that the balls $B_{2^{-k}}(\xi_k)$ are mutually disjoint. Set 
\[
S = \{h \in \R^n| \, \tfrac{h}{\bet{h}} \in \bigcup_{k\geq 1} I_k\},
\]
and
\[
K(h) = \Ind{S}(h) \bet{h}^{-n-\alpha}
\]
for $\bet{h}\ne 0$ and any fixed $\alpha \in (0,2)$. Then it is shown in \cite{BoSz05} that solutions $u$ to $Lu=0$ with $L$ as in \eqref{IntroEq:LinearK} do not need to satisfy a Harnack inequality.  This set $S$ allows to find a sequence of sets $A_m$ and also a sequence of points $x_m$ such that the probability of exiting $B_1$ and landing in $A_m$ from $x_m$ is significantly different than starting at $0$.  Specifically, if $X$ is the stochastic process generated by this particular $L$, $\tau_{B_1}$ is the exit time from $B_1$, and 
\[
u_m(x)=\P^x(X_{\tau_{B_1}}\in A_m),
\]
then one obtains 
\[
Lu_m=0\ \text{in}\ B_1
\]
and
\begin{equation*}
	\frac{u_m(0)}{u_m(x_m)}\to\infty\ \text{as}\ m\to\infty,
\end{equation*}
for an appropriately chosen sequence $x_m$.
The main ideas behind the construction are similar to those of the counterexample for the case of singular measures presented in \cite[Section 3]{BassChen-2010HarmonicStableLike}. 
\end{ex}

In light of \autoref{AppendixExample:HarnackFail}, it seems pertinent to point out where the proof of the Harnack inequality for $\A_{CS}$ (see (\ref{AppendixEq:ACS}) below) from \cite[Theorem 11.1]{CaSi-09RegularityIntegroDiff} fails for $\A_{sec}$.  A key step in \cite[Theorem 11.1]{CaSi-09RegularityIntegroDiff} is to find an equation satisfied by 
\[
w(x) = (cu(x_0)-u(x))^+,
\]
for some appropriate $c$, given that $M^-_{\A_{CS}}u\leq 1$ in $B_1$ (plus some other properties).  The details for the arguments we focus on appear in \cite[p.629,630]{CaSi-09RegularityIntegroDiff}.  In the top half of p.630, a bound at one point, $x_1$, is found on $M^-_{\A_{CS}}u(x_1)$, and then it is used via a shift in the integration variables to estimate $M^-_{\A_{CS}}w(x)$ at a different $x$.  In order that the steps on p.630 would hold for the case of $K\in\A_{sec}$, one would basically need an estimate of the form
\[
\frac{K(z-\hat x)}{K(z)}\leq \tilde c r^{-n-\al},\ \text{for all}\ z\in \real^n\setminus B_r(\hat z),
\]
where $\hat x \not = \hat z$ are certain special points.  For $K\in\A_{sec}$, this will not hold in general because there is no way to rule out the possibility that $K(z)=0$ when $K(z-\hat x)\not =0$, and thus the argument in \cite[p.630]{CaSi-09RegularityIntegroDiff} will not carry over to the case of $\A_{sec}$.  Given \autoref{AppendixExample:HarnackFail}, it is fair to say this obstruction is permanent.

\subsection{Pointwise evaluation for general $F$}\label{sec:PointwiseExamples}

We discuss in this section what one can expect from the regularity imposed on subsolutions to e.g. (\ref{IntroEq:FullyNonlinear}) by the definition of viscosity solutions at points where they can be touched from above by a test function.  It is not true in general that the very convenient result, \cite[Lemma 3.3]{CaSi-09RegularityIntegroDiff}, which gives
\begin{equation}\label{AppendixEq:DeluIntegrable}
	\int_{\real^n}\abs{\del^2_h u(x)}\abs{h}^{-n-\al}dh < \infty,
\end{equation}
still holds when one considers classes of $K$ more general than those of \cite{CaSi-09RegularityIntegroDiff}, such as $\A_{sec}$.
The formula (\ref{AppendixEq:DeluIntegrable}) seems to match with what is expected in the second order case, but is not a good guide for the integro-differential case.  Indeed let $\M^+$ be the Pucci maximal operator of the second order theory (see \cite[Chapter 2]{CaCa-95}) and in the viscosity sense (for some bounded domain, $\Omega$)
\[
\M^+u\geq -f\ \textnormal{in}\ \Omega.
\]
Then whenever $u-\phi$ has a global max at $x\in\Omega$, we have for the eigenvalues of $D^2u(x)$, $\{e_1,\dots,e_n\}$,
\begin{align}
	\lam\sum_{e_i<0} e_i &+ \Lam\sum_{e_j>0} e_j \geq -f(x)\nonumber\\
	\intertext{which implies}
	f(x) + \Lam\norm{\phi}_{C^2} &\geq \lam\sum_{e_i<0} (-e_i).\label{AppendixEq:SecondOrderBounds}
\end{align} 
From this we deduce that $u$ is also $C^{1,1}$ from below at $x$ with a bound which depends on $f$, $\Lam$, $\phi$.  However, as we show, even an analog of this to an integrated quantity such as (\ref{AppendixEq:DeluIntegrable}) is too much to ask in the general integro-differential setting.  The more generic analog to (\ref{AppendixEq:SecondOrderBounds}) is \autoref{cor:DeltaUMinusIntegrable} below.

To make this precise, let 
\begin{equation}\label{AppendixEq:ACS}
	\A_{CS} = \left\{ K:\real^n\to\real\ :\ \frac{\lam(2-\al)}{\abs{h}^{-n-\al}}\leq K(h)\leq \frac{\Lam(2-\al)}{\abs{h}^{-n-\al}}\right\}.
\end{equation}
Then without assuming $u$ satisfies any equation, we have the more straightforward observation
\begin{note}\label{note:CSGoodIntegral}
	If $\phi\in C^{1,1}(x)\cap L^\infty(\real^n)$, $u-\phi$ has a global max at $x$, and $M^+_{\A_{CS}}u(x)\geq0$ classically, then
	\begin{equation*}
		\int_{\real^n}(\del^2_h u(x))^-\abs{h}^{-n-\al}dh \leq \frac{\Lam}{\lam}\int_{\real^n}(\del^2_h\phi(x))^+\abs{h}^{-n-\al}dh \leq \frac{C(n,\al)\Lam}{\lam}C(\phi).
	\end{equation*}
\end{note}

We now present an example which illustrates the failure of \autoref{note:CSGoodIntegral} and puts in contrast $M^+_{\A_{CS}}$ to $M^+_{\A_{sec}}$.
\begin{ex}\label{AppendixExample:IntegrabilityFail}
	Let $n=2$ and $\al>1$.  Define $u$ as 
	\begin{equation}\label{AppendixEq:ExampleDef}
		u(x) = -\abs{x}\phi\left(\theta\left(\frac{x}{\abs{x}}\right)\right).
	\end{equation}
	Here $\theta(\frac{x}{\abs{x}})\in[-\pi,\pi)$ is the angle of $\frac{x}{\abs{x}}$, and  $\phi$ is a smooth \emph{even} angular cutoff such that
	\begin{equation*}
		\phi(\theta)=
		\begin{cases}
			1 &\text{if}\ \abs{\theta}\in[0,\frac{\pi}{6}]\union [\frac{5\pi}{6},\pi]\\
			\text{smooth and monotone}\ &\text{if}\ \abs{\theta}\in[\frac{\pi}{6},\frac{\pi}{4}]\union [\frac{3\pi}{4},\frac{5\pi}{6}]\\
			0 &\text{if}\ \abs{\theta}\in [\frac{\pi}{4},\frac{3\pi}{4}].
		\end{cases}
	\end{equation*}
	Let $e_1,e_2$ be the canonical basis vectors in $\real^2$ and  $V_1$, $V_2$ be the sectors
	\begin{equation*}
	V_1 = \left\{ z\ :\ \abs{\ska{\frac{z}{\abs{z}},e_1}}\geq\frac{\sqrt{3}}{2} \right\}\ \text{and}\	V_2 = \left\{ z\ :\ \abs{\ska{\frac{z}{\abs{z}},e_2}}\geq\frac{\sqrt{3}}{2} \right\}.
	\end{equation*}
	We now list-- but leave to the reader to check-- some illustrative properties of $u$:
	
	\begin{itemize}
		\item $\displaystyle u-0\ \text{has a global max at}\ x=0$
		\item $\displaystyle M^+_{\A_{sec}}u(0) = 0\ \text{by using}\ K(h)=\Indicator_{V_2}(h)\abs{h}^{-2-\al}$
		\item $\displaystyle M^-_{\A_{sec}}u(0) = -\infty\ \text{by using}\ K(h)=\Indicator_{V_1}(h)\abs{h}^{-2-\al}$
		\item $\displaystyle \int_{\real^2}(\del^2_h u(x))^- \abs{h}^{-2-\al}dh= +\infty$.
	\end{itemize}
	
\end{ex}

We now state the pointwise evaluation results of \autoref{sec:subsecPointWiseEvaluation} for a general class of kernels.  We hope these results can be useful elsewhere in the theory.  Assume that $\A$ is a class of kernels such that at least
\begin{equation}\label{AppendixEq:KernelUpperBound}
	K(h)\leq \Lam\abs{h}^{-n-\al}\ \text{for all}\ h\in\real^n.
\end{equation}
Then it is immediate that:

\begin{note}\label{AppendixNote:UniformlyIntegrable}
	If $\A$ satisfies (\ref{AppendixEq:KernelUpperBound}) then for each $R>0$ $\displaystyle \left\{\min(\abs{h}^2,1)K(h)\ :\ K\in \A\right\}$ is a uniformly integrable family of kernels on $B_R$.
\end{note}

Given \autoref{AppendixNote:UniformlyIntegrable}, the proof of \autoref{prop:PointwiseEvaluation} (via \autoref{lem:FormulaForMPlus}) can be directly adapted to the case of $M^+_{\A}$. Furthermore, the pointwise evaluation also holds for any $F$ which can be written as an $\inf-\sup$ (\autoref{prop:GeneralPointwiseEvalF}).

Indeed, if $K^*_m$ are optimizers (or achieve within $\ep$ of the supremum) for $v_m$ (as in \autoref{prop:PointwiseEvaluation}) then by \autoref{AppendixNote:UniformlyIntegrable} there is a subsequence of $\min(\abs{h}^2,1)K^*_m(h)$ which converges pointwise a.e. $x$ to some $K\in\A$.  Then we can look at the pointwise convergence of
\[
\frac{\del^2_h v_m(x)}{\min(\abs{h}^2,1)}\min(\abs{h}^2,1)K^*_m(h)
\]
to conclude.  This discussion proves \autoref{prop:GeneralPointwiseEvalMPlus}, and minor modifications yield \autoref{prop:GeneralPointwiseEvalF}.

\begin{prop}\label{prop:GeneralPointwiseEvalMPlus}
	Assume that all $K\in\A$ satisfy (\ref{AppendixEq:KernelUpperBound}). Then the assertion of \autoref{prop:PointwiseEvaluation} remains true for $M^+_{\A_{sec}}$ replaced by $M^+_\A$.
\end{prop}

\begin{prop}\label{prop:GeneralPointwiseEvalF}
	Assume that $F$ is given by
	\begin{equation*}
		F(u,x) = \inf_{a\in \S} \sup_{K\in\A_a}\left\{\int_{\real^n}\del^2_h u(x)K(h)dh\right\},
	\end{equation*}
	where $\S$ is an arbitrary index set, and $\A_a\subset\A$ for all $a\in\S$.
	If $u$ is a viscosity subsolution of $F(u,x)\geq f(x)$ and $u-\phi$ has a global maximum at $x$, then $F(u,x)$ is defined classically, and $F(u,x)\geq f(x)$. 
\end{prop}

A useful consequence of \autoref{prop:GeneralPointwiseEvalMPlus} is:
\begin{cor}\label{cor:DeltaUMinusIntegrable}
	If $u$ is a viscosity subsolution of $F(u,x)\geq -f(x)$ in $\Omega$ and $u-\phi$ has a global maximum at $x\in\Omega$, then there exists at least one $K^*\in\A$ such that $(\del^2_h u)^-$ is integrable against $K^*$.  Furthermore,
	\begin{equation}
		\int_{\real^n}(\del^2_h u(x))^-K^*(h)dh \leq f(x)+\Lam\sup_{K\in\A}\left(\int_{B_1}(\del^2_h \phi(x,y))^+K(h)dh\right) + 1.
	\end{equation}
\end{cor}


\bibliographystyle{plain}
\bibliography{NewKernel_Master.bib}

\begin{thebibliography}{10}

\bibitem{BaChIm-11Holder}
G.~Barles, E.~Chasseigne, and C.~Imbert.
\newblock H\"older continuity of solutions of second-order elliptic
  integro-differential equations.
\newblock {\em J. Eur. Math. Soc.}, 13(1):1--26, 2011.

\bibitem{BaIm-07}
Guy Barles and Cyril Imbert.
\newblock Second-order elliptic integro-differential equations: viscosity
  solutions' theory revisited.
\newblock {\em Ann. Inst. H. Poincar\'e Anal. Non Lin\'eaire}, 25(3):567--585,
  2008.

\bibitem{Bass2009-RegularityStableLikeJFA}
Richard~F. Bass.
\newblock Regularity results for stable-like operators.
\newblock {\em J. Funct. Anal.}, 257(8):2693--2722, 2009.

\bibitem{BassChen-2010HarmonicStableLike}
Richard~F. Bass and Zhen-Qing Chen.
\newblock Regularity of harmonic functions for a class of singular stable-like
  processes.
\newblock {\em Math. Z.}, 266(3):489--503, 2010.

\bibitem{BaKa-05Holder}
Richard~F. Bass and Moritz Kassmann.
\newblock H\"older continuity of harmonic functions with respect to operators
  of variable order.
\newblock {\em Comm. Partial Differential Equations}, 30(7-9):1249--1259, 2005.

\bibitem{BaLe02}
Richard~F. Bass and David~A. Levin.
\newblock Harnack inequalities for jump processes.
\newblock {\em Potential Anal.}, 17(4):375--388, 2002.

\bibitem{BCF12}
C.~Bjorland, L.~Caffarelli, and A.~Figalli.
\newblock Non-local gradient dependent operators.
\newblock {\em Adv. Math.}, 230(4-6):1859--1894, 2012.

\bibitem{BoSz05}
Krzysztof Bogdan and Pawe{\l} Sztonyk.
\newblock Harnack's inequality for stable {L}\'evy processes.
\newblock {\em Potential Anal.}, 22(2):133--150, 2005.

\bibitem{CaSi-09RegularityIntegroDiff}
Luis Caffarelli and Luis Silvestre.
\newblock Regularity theory for fully nonlinear integro-differential equations.
\newblock {\em Comm. Pure Appl. Math.}, 62(5):597--638, 2009.

\bibitem{CaSi-09EvansKrylov}
Luis Caffarelli and Luis Silvestre.
\newblock The {E}vans-{K}rylov theorem for nonlocal fully nonlinear equations.
\newblock {\em Ann. of Math. (2)}, 174(2):1163--1187, 2011.

\bibitem{CaSi-09RegularityByApproximation}
Luis Caffarelli and Luis Silvestre.
\newblock Regularity results for nonlocal equations by approximation.
\newblock {\em Arch. Ration. Mech. Anal.}, 200(1):59--88, 2011.

\bibitem{CaCa-95}
Luis~A. Caffarelli and Xavier Cabr{\'e}.
\newblock {\em Fully nonlinear elliptic equations}, volume~43 of {\em American
  Mathematical Society Colloquium Publications}.
\newblock American Mathematical Society, Providence, RI, 1995.

\bibitem{Chan-2012NonlocalDriftArxiv}
H{\'e}ctor Chang~Lara.
\newblock Regularity for fully non linear equations with non local drift.
\newblock {\em arXiv:1210.4242 [math.AP]}, 2012.

\bibitem{ChDa-2012NonsymKernels}
H{\'e}ctor Chang~Lara and Gonzalo D{\'a}vila.
\newblock Regularity for solutions of nonlocal, nonsymmetric equations.
\newblock {\em Ann. Inst. H. Poincar\'e Anal. Non Lin\'eaire}, 29(6):833--859,
  2012.

\bibitem{ChDa-2012RegNonlocalParabolicCalcVar}
Héctor Chang~Lara and Gonzalo Dávila.
\newblock Regularity for solutions of non local parabolic equations.
\newblock {\em Calc. Var. PDE}, 2012.
\newblock published online.

\bibitem{CrIsLi-92}
Michael~G. Crandall, Hitoshi Ishii, and Pierre-Louis Lions.
\newblock User's guide to viscosity solutions of second order partial
  differential equations.
\newblock {\em Bull. Amer. Math. Soc. (N.S.)}, 27(1):1--67, 1992.

\bibitem{GiTr-98}
David Gilbarg and Neil~S. Trudinger.
\newblock {\em Elliptic partial differential equations of second order}.
\newblock Classics in Mathematics. Springer-Verlag, Berlin, 2001.
\newblock Reprint of the 1998 edition.

\bibitem{GuSc12}
Nestor Guillen and Russell~W. Schwab.
\newblock Aleksandrov-{B}akelman-{P}ucci type estimates for
  integro-differential equations.
\newblock {\em Arch. Ration. Mech. Anal.}, 206(1):111--157, 2012.

\bibitem{IshiiLions-1990ViscositySolutions2ndOrder}
H.~Ishii and P.-L. Lions.
\newblock Viscosity solutions of fully nonlinear second-order elliptic partial
  differential equations.
\newblock {\em J. Differential Equations}, 83(1):26--78, 1990.

\bibitem{MiKa12}
Moritz Kassmann and Ante Mimica.
\newblock Analysis of jump processes with nondegenerate jumping kernels.
\newblock {\em Stochastic Process. Appl.}, 123(2):629--650, 2013.

\bibitem{KassSchwa-2013RegularityNonlocalParaArXiv}
Moritz Kassmann and Russell~W. Schwab.
\newblock Regularity results for nonlocal parabolic equations.
\newblock {\em arxiv.org}, 2013.

\bibitem{Kriventsov-2013RegRoughKernelsArXiv}
Dennis Kriventsov.
\newblock $ {C}^{1, \alpha}$ interior regularity for nonlinear nonlocal
  elliptic equations with rough kernels.
\newblock {\em arXiv preprint arXiv:1304.7525}, 2013.

\bibitem{Krylov-1980ControlledDiffusionProcesses}
N.~V. Krylov.
\newblock {\em Controlled diffusion processes}, volume~14 of {\em Stochastic
  Modelling and Applied Probability}.
\newblock Springer-Verlag, Berlin, 2009.
\newblock Translated from the 1977 Russian original by A. B. Aries, Reprint of
  the 1980 edition.

\bibitem{Land-72}
N.~S. Landkof.
\newblock {\em Foundations of modern potential theory}.
\newblock Springer-Verlag, New York, 1972.
\newblock Translated from the Russian by A. P. Doohovskoy, Die Grundlehren der
  mathematischen Wissenschaften, Band 180.

\bibitem{RangThesis}
Marcus Rang.
\newblock Regularity results for nonlocal fully nonlinear elliptic equations.
\newblock PhD thesis, Bielefeld University, 2013.

\bibitem{Schw-10Per}
Russell~W. Schwab.
\newblock Periodic homogenization for nonlinear integro-differential equations.
\newblock {\em SIAM J. Math. Anal.}, 42(6):2652--2680, 2010.

\bibitem{Schw-12StochCPDE}
Russell~W. Schwab.
\newblock Stochastic homogenization for some nonlinear integro-differential
  equations.
\newblock {\em Communications in Partial Differential Equations},
  38(2):171--198, 2012.

\bibitem{silvestre}
Luis Silvestre.
\newblock H\"older estimates for solutions of integro-differential equations
  like the fractional {L}aplace.
\newblock {\em Indiana Univ. Math. J.}, 55(3):1155--1174, 2006.

\bibitem{SoVo04}
Renming Song and Zoran Vondra{\v{c}}ek.
\newblock Harnack inequality for some classes of {M}arkov processes.
\newblock {\em Math. Z.}, 246(1-2):177--202, 2004.

\end{thebibliography}

\end{document}